\documentclass[11pt]{amsart}
\usepackage[parfill]{parskip}    

\usepackage{graphicx, txfonts} 

\usepackage{hyperref}

\usepackage{amsthm, amssymb, amsfonts, amsmath, enumerate}
\usepackage{mathrsfs} 
\usepackage{color}

\usepackage{xy}
\input xy
\xyoption{all}
\newtheorem{main}{Theorem}[section]

\newtheorem{theorem}{Theorem}[section]
\newtheorem{proposition}[theorem]{Proposition}

\newtheorem{lemma}[theorem]{Lemma}
\newtheorem{corollary}[theorem]{Corollary}

\theoremstyle{definition}
\newtheorem{definition}[theorem]{Definition}
\newtheorem{defn}[theorem]{Definition}

\newtheorem{example}[theorem]{Example}

\theoremstyle{remark}
\newtheorem{remark}[theorem]{Remark}

\numberwithin{equation}{section}

\DeclareMathOperator{\inj}{inj}
\DeclareMathOperator{\cell}{cell}
\DeclareMathOperator{\cof}{cof}

\newcommand{\M}{\mathcal{M}}

\newcommand{\sF}{\mathscr{F}}
\newcommand{\sW}{\mathscr{W}}
\newcommand{\sQ}{\mathscr{Q}}

\newcommand{\algo}{\mbox{Alg}_{\cat O}}

\renewcommand{\emptyset}{\varnothing}

\renewcommand{\tilde}[1]{\widetilde{#1}}

\DeclareMathOperator{\dgcat}{dgcat}

\DeclareMathOperator{\colim}{colim}

\DeclareMathOperator{\hocolim}{hocolim}

\newcommand{\calm}{\mathcal{M}}
\newcommand{\calc}{\mathcal{C}}
\newcommand{\lcm}{L_{\mathcal{C}}\calm}

\newcommand{\po}{\ar@{}[dr]|(.7){\Searrow}}
\newcommand{\pb}{\ar@{}[dr]|(.3){\Nwarrow}}

\newcommand{\cat}[1]{\mathcal{#1}}

\title{Left Bousfield localization without left properness}
\author{Michael Batanin}
\address{Institute of Mathematics of Czech Academy of Sciences, Zitna 25, Prague 1, The Czech Republic, and MFF UK, Sokolovsk\'{a} 83, 186 75 Prague 8, The Czech Republic}
\email{bataninmichael@gmail.com}
\thanks{The first author gratefully acknowledges the financial support of the IHES in Paris, Praemium Academiae of M. Markl, RVO:67985840, and the grant GA\v{C}EXPRO19-28628X}

\author{David White}
\address{Department of Mathematics \\ Denison University
\\ Granville, OH 43023}
\email{david.white@denison.edu}
\thanks{The second author was supported by the National Science Foundation under Grant No. IIA-1414942, and by the Australian Category Seminar of Macquarie University.}

\begin{document}

\maketitle

\begin{abstract}
Given a combinatorial (semi-)model category $M$ and a set of morphisms $\cat C$, we establish the existence of a semi-model category $L_{\cat C} M$ satisfying the universal property of the left Bousfield localization in the category of semi-model categories. Our main tool is a semi-model categorical version of a result of Jeff Smith, that appears to be of independent interest. Our main result allows for the localization of model categories that fail to be left proper. We give numerous examples and applications, related to the Baez-Dolan stabilization hypothesis, localizations of algebras over operads, chromatic homotopy theory, parameterized spectra, $C^*$-algebras, enriched categories, dg-categories, functor calculus, and Voevodsky's work on radditive functors.
\end{abstract}

\section{Introduction}

Left Bousfield localization is a fundamental tool in abstract homotopy theory. It is used for the study of homology localizations of spaces and spectra \cite{bousfield-localization-spaces-wrt-homology, bous79}, the existence of stable model structures for (classical, equivariant, and motivic) spectra \cite{hovey-spectra}, the towers used in chromatic homotopy theory \cite{rav84}, computations in equivariant homotopy theory \cite{gutierrez-white-equivariant}, computations in motivic homotopy theory \cite{grso}, the study of recollement \cite{gillespie-survey}, in homological algebra \cite{bauer}, representation theory \cite{hovey-cotorsion}, universal algebra \cite{white-yau}, graph theory \cite{deb}, Goodwillie calculus \cite{chorny-white, pereira}, the homotopy theory of homotopy theories \cite{bergner, Rezk}, and the theory of higher categories \cite{Reedy-paper}, among many other applications.

The left Bousfield localization of a model category $\calm$ relative to a class of morphisms $\calc$ is a model structure $\lcm$ on the category $\calm$, where the morphisms in $\calc$ are contained in the weak equivalences in $\lcm$, and the identity functor $Id:\calm \to \lcm$ satisfies the universal property that, for any model category $\cat{N}$, any left Quillen functor $F:\calm \to \cat{N}$, taking the morphisms in $\calc$ to weak equivalences in $\cat{N}$, factors through $\lcm$. Normally, to prove that $\lcm$ exists one requires $\calc$ to be a set, and $\calm$ to be left proper and cellular \cite{hirschhorn} or left proper and combinatorial \cite{barwickSemi, beke}. In this paper, we demonstrate that, even without left properness, $\lcm$ still exists as a semi-model category, and satisfies the universal property in the category of semi-model categories. This answers a question of Barwick \cite[Remark 4.13]{barwickSemi}, was also known to Cisinski (private correspondence), and has been already applied in \cite{Reedy-paper} and \cite{white-oberwolfach}. Our work closely parallels Barwick's approach \cite{barwickSemi}, and also builds on two previous papers that worked out special cases of left Bousfield localization without left properness in cellular settings (using very different techniques), namely \cite[Theorem 1.5.1]{goerss-hopkins} (homology localization in operad-algebras in spectra) and \cite[Theorem 5.14]{harper-zhang} ($TQ$-homology localization in operad algebras in spectra). Recently, an entirely different approach to localization without left properness has been discovered by Simon Henry \cite{henry}.

Semi-model categories were introduced in \cite{hovey-monoidal} and \cite{spitzweck-thesis} in the context of algebras over operads, and are reviewed in Definition \ref{defn:semi}. A semi-model category satisfies axioms similar to those of a model category, but one only knows that morphisms {\em with cofibrant domains} admit a factorization into a trivial cofibration followed by a fibration, and one only knows that trivial cofibrations {\em with cofibrant domains} lift against fibrations. Hence, on the subcategory of cofibrant objects, a semi-model category behaves exactly like a model category, and every semi-model category admits a functorial cofibrant replacement functor. Consequently, every result about model categories has a version for semi-model categories, usually obtained by cofibrantly replacing objects as needed. This includes the usual characterization of morphisms in the homotopy category, Quillen pairs, simplicial mapping spaces, hammock localization, path and cylinder objects, Ken Brown's lemma, the retract argument, the cube lemma, projective/injective/Reedy semi-model structures, latching and matching objects, cosimplicial and simplicial resolutions, computations of homotopy limits and colimits, and more. In practice, a semi-model structure is just as useful as a full model structure.

The main source of examples of semi-model categories arises from the theory of transferred (also known as left-induced) model structures. If $T$ is a monad, the transferred structure, on the category of $T$-algebras in a model category $\calm$, defines weak equivalences and fibrations to be created and reflected by the forgetful functor to $\calm$. When $T$ arises from an operad, this transferred structure is commonly a semi-model structure \cite{grso, fresse-book, white-yau}, but is not always a full model structure \cite[Example 2.9]{batanin-white-eilenberg-moore}.
Semi-model categories have been used to prove important results all over homotopy theory \cite{barwickSemi, florian, batanin-baez-dolan-via-semi, bdw1, bdw2, batanin-white-baez-dolan, batanin-white-eilenberg-moore, EKMM, fresse-book, goerss-hopkins, grso, gutierrez-white-equivariant, harper-zhang, hovey-monoidal, mandell, nuiten, ostvaer, spitzweck-thesis, white-thesis, white-localization, white-yau, white-yau-coloc, white-yau3, white-yau4, white-yau6, yau-book}.

In recent years, the authors have seen a large number of cases where one wishes to left Bousfield localize a model structure that is not known to be left proper \cite{bacard, bacard-published, tillmann1, beardsley, bergner, lazarev, goerss-hopkins, hackney, johnson, kk, pereira, richter, tabuada, toen, deb, voevodsky}. Our main result provides a way to do this, and even to left Bousfield localize semi-model structures. We now state our main result.

\begin{main} \label{main}
Suppose that $\M$ is a combinatorial semi-model category whose generating cofibrations have cofibrant domains, and $\cat C$ is a set of morphisms of $\M$. Then there is a semi-model structure $L_{\cat C}(\M)$ on $\M$, whose weak equivalences are the $\cat C$-local equivalences, whose cofibrations are the same as $\M$, and whose fibrant objects are the $\cat C$-local objects. Furthermore, $L_{\cat C}(\M)$ satisfies the universal property that, for any any left Quillen functor of semi-model categories $F:\M\to \cat{N}$ taking $\cat C$ into the weak equivalences of $N$, then $F$ is a left Quillen functor when viewed as $F:L_{\cat C}(\M)\to \cat{N}$.
\end{main}

Note that, if $\M$ is a model category, then $\M$ is automatically a semi-model category, and so the theorem above proves that left Bousfield localization $L_{\cat C}\M$ exists (as a semi-model category) for non-left proper model categories $\M$.

Our main tool to prove Theorem \ref{main} is a semi-model categorical version of a famous theorem of Jeff Smith (\cite[Proposition 2.2]{barwickSemi}, \cite[Theorem 1.7]{beke}). We apply this theorem by taking $\sW$ to be the class of $\cat C$-local equivalences.

\begin{main} \label{main-smith}
Suppose $\M$ is a locally presentable category with a class $\sW$ of weak equivalences and a set of morphisms $I$ satisfying
\begin{enumerate}
\item The class of weak equivalences with cofibrant domains is $\kappa$-accessible.
\item The class $\sW$ is closed under retracts and the two out of three property.
\item Any morphism in inj$(I)$ is a weak equivalence.
\item Within the class of trivial cofibrations, defined to be the intersection of $\cof I$ and $\sW$, morphisms with cofibrant domains are closed under pushouts to arbitrary cofibrant objects and under transfinite composition.
\item The morphisms of $I$ have cofibrant domains.
\end{enumerate}
Then there is a cofibrantly generated semi-model structure on $\M$ with generating cofibrations $I$, generating trivial cofibrations $J$, cofibrations $\cof I$, and fibrations defined by the right lifting property with respect to $J$. Furthermore, the generating trivial cofibrations $J$ have cofibrant domains.
\end{main}

After a review of the main definitions in Section \ref{sec:preliminaries}, we prove Theorem \ref{main-smith} in Section \ref{sec:smith}. We then prove Theorem \ref{main} in Section \ref{sec:bousfield}. As the main value of our approach is that we do not need $\M$ to be left proper in order for its localization to exist, we now explain the key idea that allows this assumption to be avoided. A model category is left proper if any pushout of a weak equivalence $f: A\to B$ along a cofibration $g: A\to C$ is a weak equivalence $h: C\to P$. The semi-model category version of this statement assumes that $f$ is a weak equivalence \textit{between cofibrant objects}. With this extra assumption, $h$ is always a weak equivalence, so left properness is automatic. 

The main place where left properness is used when proving the existence of left Bousfield localization, is to prove that pushouts of trivial cofibrations are again trivial cofibrations (see Chapter 3 of \cite{hirschhorn}, and note that left properness is not required till Proposition 3.2.10). Crucially, in a left proper model category, a pushout square where one leg is a cofibration is a homotopy pushout square \cite[Proposition 1.19]{barwickSemi}. Thankfully, when we establish a semi-model structure $L_{\cat C}\M$, we only need this for a pushout square where all objects are cofibrant, and one leg is a trivial cofibration. Such squares are always homotopy pushout squares, even when $L_{\cat C}\M$ is only a semi-model structure.

The other main place where left properness is needed in the theory of left Bousfield localization is Proposition 13.2.1 in \cite{hirschhorn}, which states that for any cofibration $g: A\to B$, any fibration $p: X\to Y$, and any cofibrant replacement $Qg: QA\to QB$ (which is a cofibration) as shown below:
\[
\xymatrix{QA \ar@{^(->}[d] \ar[r] & A \ar@{^(->}[d] \ar[r] & X \ar@{->>}[d]  \\  
QB \ar[r] & B \ar[r] & Y}
\]
then $p$ has the right lifting property with respect to $g$, if $p$ has the right lifting property with respect to $Qg$ and if $\M$ is left proper. This is used when characterizing the fibrant objects of a left Bousfield localization, and when verifying the universal property of left Bousfield localization, because of the way Hirschhorn defines his set of generating trivial cofibrations \cite[Definition 4.2.2]{hirschhorn}. The semi-model category version of \cite[Proposition 13.2.1]{hirschhorn} assumes $g$ is already a cofibration between cofibrant objects, and hence holds in any semi-model category, or in any model category (even one that fails to be left proper). In our case, we side-step this result entirely, because for us, the domains of the generating (trivial) cofibrations in $L_{\cat C} \M$ are cofibrant, and the local fibrations are \textit{defined} to be morphisms with the right lifting property with respect to the generating trivial cofibrations $J_{\cat C}$ provided by Theorem \ref{main-smith}. 

After proving Theorem \ref{main} in Section \ref{sec:bousfield}, in Section \ref{sec:applications}, we consider several applications of Theorem \ref{main} and  propose future directions.

We open Section \ref{sec:applications} by recalling  Voevodsky's theory of radditive functors \cite{voevodsky}. Voevodsky constructs an example \cite[Example 3.48]{voevodsky}  of a non left proper  category of radditive functors and proves that it does not admit a left Bousfield localization as a model category. This example clearly shows that our theory of semi-model localization is a powerful new tool which allows us to overcome many technical difficulties arising from the non-existence of model theoretical localization. 

We then continue in Section \ref{sec:applications} with sample applications from different areas of homotopy theory.  We briefly describe the results from our companion paper \cite{Reedy-paper}, where we prove a strong version of the Baez-Dolan stabilization hypothesis \cite{baez-dolan} for Rezk's model of weak $n$-categories \cite{Rezk}. Other applications of our result include  the resolution model structure in chromatic homotopy theory \cite{goerss-hopkins}, $TQ$-homology \cite{harper-zhang}, Ravenel's $X(n)$-spectra \cite{beardsley}, parameterized spectra after Intermont and Johnson \cite{johnson}, $C^*$-algebras \cite{kk,ostvaer}, chain complexes \cite{richter}, inverting operations in ring theory \cite{lazarev} and operad theory \cite{hackney}, the theory of weakly enriched categories \cite{bacard-published}, dg-categories \cite{toen}, Goodwillie calculus \cite{pereira}, graph theory \cite{deb} and the theory of homotopy colimits of diagrams of model categories \cite{bergner}. We anticipate many more applications of Theorem \ref{main} in the years to come.

\section*{Acknowledgments}

The authors would like to thank Denis-Charles Cisinski for suggesting this problem, and Clark Barwick for leaving such a nice road-map to its resolution. The second author is grateful to Mark Johnson for suggesting Example \ref{ex:intermont}, to Simon Henry for sharing an advance version of his work, and to Valery Isaev and Truong Hoang Manh for a careful reading of the first arXiv version. We would also like to thank Macquarie University for hosting the second author on three occasions while we carried out this research, and we thank the anonymous referee for comments that improved the exposition.

\section{Preliminaries} \label{sec:preliminaries}

In this section, we recall definitions and useful results about semi-model categories, and about left Bousfield localization. For further details on these topics, we refer the reader to \cite{barwickSemi, fresse-book, goerss-hopkins, hovey-monoidal, spitzweck-thesis, white-localization, white-yau} and to \cite{hirschhorn}. We assume the reader is familiar with the basics of model categories, as recounted in \cite{hovey-book}. We begin with the definition of a semi-model category \cite{barwickSemi}. Recall that, for a set of morphisms $S$, $\inj S$ refers to the class of morphisms having the right lifting property with respect to $S$.

\begin{defn} \label{defn:semi}
A \textit{semi-model structure} on a category $\M$ consists of classes of weak equivalences $\sW$, fibrations $\sF$, and cofibrations $\sQ$ satisfying the following axioms:

\begin{enumerate}
\item[M1] Fibrations are closed under pullback.
\item[M2] The class $\sW$ is closed under the two out of three property.
\item[M3] $\sW,\sF,\sQ$ contain the isomorphisms, are closed under composition, and are closed under retracts.
\item[M4] 
\begin{enumerate}
\item[i] Cofibrations have the left lifting property with respect to trivial fibrations.
\item[ii] Trivial cofibrations whose domain is cofibrant have the left lifting property with respect to fibrations.
\end{enumerate}
\item[M5] 
\begin{enumerate}
\item[i] Every morphism in $\M$ can be functorially factored into a cofibration followed by a trivial fibration. 
\item[ii] Every morphism whose domain is cofibrant can be functorially factored into a trivial cofibration followed by a fibration.
\end{enumerate} 
\end{enumerate}

If, in addition, $\M$ is bicomplete, then we call $\M$ a \textit{semi-model category}.
$\M$ is said to be \textit{cofibrantly generated} if there are sets of morphisms $I$ and $J$ in $\M$ such that $\inj I$ is the class of trivial fibrations, $\inj J$ is the class of fibrations in $\M$, the domains of $I$ are small relative to $I$-cell, and the domains of $J$ are small relative to morphisms in $J$-cell whose domain is cofibrant. We will say $\M$ is \textit{combinatorial} if it is cofibrantly generated and locally presentable.
\end{defn}

Our definition of semi-model category follows Barwick \cite{barwickSemi} (taking $E = C$ there), which was inspired by Spitzweck's notion of a $J$-semi model category \cite{spitzweck-thesis}, but removing the need for this abstract structure to be transferred from some underlying model category. Many of the semi-model categories $\M$ that we have in mind are in fact transferred along a right adjoint $U: \M \to \cat D$, so that weak equivalences and fibrations in $\M$ are morphisms $f$ such that $U(f)$ is a weak equivalence or fibration in $\cat D$. But the definition allows for semi-model categories to exist without reference to a model category $\cat D$, and Barwick showed how to recover Spitzweck's results in this more general setting \cite{barwickSemi}. Although Barwick and Spitzweck originally included an axiom that the initial object is cofibrant, this axiom is redundant. The statement can easily be deduced from a factorization, lifting, and retract argument applied to the identity morphism on the initial object.

We note that Spitzweck originally assumed in (M1) that trivial fibrations are also closed under pullback. Barwick proved that this is redundant, since trivial fibrations are characterized as morphisms having the right lifting property with respect to cofibrations \cite[Lemma 1.7]{barwickSemi} and hence are closed under pullback and composition. For a cofibrantly generated semi-model category, (M1) is entirely redundant, since fibrations are characterized as $\inj J$. If the (co)domains of morphisms in $J$ (resp. $I$) are compact, these observations show that (trivial) fibrations are closed under transfinite composition. Throughout this paper, we work with cofibrantly generated semi-model categories, so we say nothing more about (M1). For cofibrantly generated semi-model categories, the first two parts of M3 are automatic.

We note that the assumptions we require of a semi-model category are stricter than those required by Fresse \cite{fresse-book}, who generalized Spitzweck's notion of an $(I,J)$-semi model structure, and hence all results proven by Fresse hold in our setting. We gather a few useful results, the proofs of which are useful exercises (which may also be found in \cite{barwickSemi, fresse-book, spitzweck-thesis}):

\begin{lemma} \label{lemma:transfinite-comp-and-pushout}
Let $\M$ be a cofibrantly generated semi-model category. Then:

\begin{enumerate}
\item Cofibrations are closed under pushout and transfinite composition.
\item A relative $J$-cell complex with cofibrant domain is a trivial cofibration.
\item Trivial cofibrations with cofibrant domains are retracts of relative $J$-cell complexes.
\end{enumerate}
\end{lemma}

Cofibrantly generated semi-model categories admit framings, and hence homotopy function complexes $map(-,-) \in sSet$, that we will use when we Bousfield localize. However, it is also possible to define a \textit{simplicial} semi-model structure (see \cite[Definition 1.1.8]{goerss-hopkins}) on a simplicial category (i.e., enriched, tensored, and cotensored over $sSet$), where one can use simplicial mapping spaces $Map(-,-)$ as homotopy function complexes. We will not use this notion, but we include a definition for the convenience of future readers.

\begin{definition}
Suppose $\M$ is a simplicial category. A semi-model structure on $\M$ is \textit{simplicial} if, whenever $i: A\to B$ is a cofibration with cofibrant domain, and $q: X\to Y$ is a fibration, then the induced map $Map(B,X)\to Map(B,Y)\times_{Map(A,Y)} Map(A,X)$ is a fibration in $sSet$ that is furthermore a weak equivalence if either $i$ or $q$ are.
\end{definition}

For simplicial semi-model categories, it is a bit easier to Bousfield localize, in practice.

\section{Smith's theorem for locally presentable semi-model categories} \label{sec:smith}

In this section, we prove a version of Smith's theorem \cite{barwickSemi, beke}, that provides a set $J$ of generating trivial cofibrations to produce a semi-model structure on a locally presentable category with a given class of weak equivalences, and a given set of generating cofibrations, satisfying some compatibility axioms. This result is our main tool for proving Theorem \ref{main}. In the following, for a class of morphisms $S$, $\cof S$ means morphisms with the left lifting property with respect to $\inj S$. We say an object $X$ is cofibrant if $\emptyset\to X$ is in $\cof I$. We let $S_c$ denote the class of morphisms in $S$ that have cofibrant domains. We refer the reader to \cite{adamek} for terminology related to accessibility.

\begin{theorem} \label{thm:smith-thm-semi-output}
Suppose $\M$ is a locally presentable category with a class $\sW$ of weak equivalences and a set of morphisms $I$ satisfying:
\begin{enumerate}
\item the class $\sW_c$ is $\kappa$-accessible,
\item the class $\sW$ is closed under retracts and satisfies the two out of three property,
\item any morphism in inj$(I)$ is a weak equivalence,
\item within the class of trivial cofibrations, defined to be the intersection of $\cof I$ and $\sW$, morphisms with cofibrant domains are closed under pushouts to arbitrary cofibrant objects and under transfinite composition, and 
\item the morphisms of $I$ have cofibrant domains.
\end{enumerate}
Then there is a combinatorial semi-model structure on $\M$ with generating cofibrations $I$, generating trivial cofibrations $J$, cofibrations $\cof I$, and fibrations defined by the right lifting property with respect to $J$. Furthermore, the generating trivial cofibrations $J$ have cofibrant domains.
\end{theorem}

We will now prove this, following \cite[Lemma 2.3 and 2.4]{barwickSemi} (equivalently, \cite[Lemma 1.8 and 1.9]{beke}), which we restate for semi-model categories below. Let $(\sW \cap \cof I)_c$ denote the subclass of $\sW \cap \cof I$ consisting of morphisms with cofibrant domains (hence cofibrant codomains as well). Let $(\cell J)_c$ denote the collection of transfinite compositions of pushouts of morphisms of $J$ along morphisms into cofibrant objects.

\begin{lemma} \label{lemma:cof-J-characterization}
Under the hypotheses of Theorem \ref{thm:smith-thm-semi-output}, suppose $J\subset (\sW\cap\cof I)_c$ is a set of morphisms such that any commutative square
\begin{equation*}
\xymatrix@C=18pt@R=18pt{
K\ar[d]_i \ar[r]&M\ar[d]^w\\
L\ar[r]&N
}
\end{equation*}
in which $i \in I$ and $w$ is in $\sW_c$, can be factored as a commutative diagram
\begin{equation*}
\xymatrix@C=18pt@R=18pt{
K\ar[d]\ar[r]&M'\ar[d]\ar[r]&M\ar[d]\\
L\ar[r]&N'\ar[r]&N,
}
\end{equation*}
in which $M'\to N'$ is in $J$. Then $(\cof J)_c = (\sW \cap \cof I)_c$.
\end{lemma}

\begin{proof}
To show $(\cof J)_c \supset (\sW \cap \cof I)_c$, let $f \in (\sW \cap \cof I)_c$, and recall that this means $f$ has cofibrant domain. We will factor $f$ as an element $i$ of $(\cell J)_c$ followed by an element $p$ of $\inj I$. Once we do this, $f$ has the left lifting property with respect to $p$ and the retract argument says $f$ is a retract of $i$. Lemma 2.1.10 in \cite{hovey-book} demonstrates that $i$ is in $\cof J$. As $\cof J$ is defined by lifting, it is closed under retract, so this proves $f$ is in $\cof J$. Since $f$ was assumed to be a morphism between cofibrant objects, $f$ is in fact in $(\cof J)_c$.

To produce the factorization for $f$ we follow \cite{barwickSemi}. Choose a regular cardinal $\kappa$ such that the codomains of morphisms in $I$ are $\kappa$-presentable. Consider the set $(I/f)$ of squares
\begin{equation*}
\xymatrix@C=18pt@R=18pt{K\ar[d]_{i}\ar[r]&X\ar[d]^f\\
L\ar[r]&Y,}
\end{equation*}
where $i\in I$; for each such square use the hypothesis to produce a factorization
\begin{equation*}
\xymatrix@C=18pt@R=18pt{
K\ar[d]_{i}\ar[r]&M(i)\ar[d]^{j_{(i,f)}}\ar[r]&X\ar[d]^f\\
L\ar[r]&N(i)\ar[r]&Y,
}
\end{equation*}
where $j_{(i,f)}\in J$, and let $M_{(I/f)}\to N_{(I/f)}$ be the coproduct $\coprod_{i\in(I/f)} \limits j_{(i,f)}$. Define an endofunctor $Q$ of $(\sW/Y)$ by
\begin{equation*}
Qf:=\left[X\coprod_{M_{(I/f)}} N_{(I/f)}\to Y\right]
\end{equation*}
for any morphism $f:X\to Y$ in $\sW_c$. For any regular cardinal $\alpha$, set $Q^{\alpha}:=\colim_{\beta<\alpha}Q^{\beta}$. This provides, for any morphism $f:X\to Y$ in $\sW_c$, a functorial factorization
\begin{equation*}
\xymatrix@1@C=18pt{X\ar[r]&Q^{\kappa}f\ar[r]&Y}
\end{equation*}
where the morphism $X\to Q^\kappa f$ is in $\cell J$ and the morphism $Q^\kappa f\to Y$ is in $\inj J$.

The containment $(\cof J)_c \subset (\sW \cap \cof I)_c$ follows from Proposition 2.1.15 in \cite{hovey-book}, from the small object argument above, and from hypothesis (4) of the theorem. This is because any morphism in $(\cof J)_c$ is a retract of a morphism in $(\cell J)_c$, via the retract argument and the factorization provided above (as well as the hypothesis that $J$ consists of cofibrations between cofibrant objects). Next, hypothesis (4) ensures that $(\cell J)_c \subset (\sW \cap \cof I)_c$, and both $\sW$ and $\cof I$ are closed under retract (the former by hypothesis (2) in Theorem \ref{thm:smith-thm-semi-output}; the latter because it is defined via a lifting property).
\end{proof}

Observe that it is not true in general for semi-model categories that trivial cofibrations are closed under transfinite composition and pushout. The class of morphisms $\cell J$ might not be contained in $\sW \cap \cof I$, even though it is always contained in $\cof J$. However, requiring the domains of the morphisms in $J$ to be cofibrant and only considering pushouts via morphisms to cofibrant objects will result in $(\cell J)_c$ being contained in $(\sW \cap \cof I)_c$ by Lemma \ref{lemma:transfinite-comp-and-pushout} (hence $(\cof J)_c\subset (\sW \cap \cof I)_c$). Similarly, Lemma \ref{lemma:transfinite-comp-and-pushout} implies that $(\cof J)_c\supset (\sW \cap \cof I)_c$.

Next we address the existence of a set $J$ which factorizes squares as above.

\begin{lemma} \label{lemma:construction-of-J}
Let $\M$ be a locally presentable category, with a set $I$ of morphisms and a class $\sW$ of morphisms (satisfying the two out of three property) such that $\sW_c$ is accessible and $\inj(I) \subset \sW$. Then there is a set $J$ satisfying the conditions of the lemma above.
\end{lemma}

\begin{proof} Suppose $i:K\to L$ is in $I$. Since $\sW_c$ is an accessibly embedded accessible subcategory of the arrow category $Arr(\M)$, there exists a subset $\sW(i)\subset \sW_c$ such that for any commutative square
\begin{equation*}
\xymatrix@C=18pt@R=18pt{
K\ar[d]_i\ar[r]&M\ar[d]\\
L\ar[r]&N
}
\end{equation*}
in which $M\to N$ is in $\sW_c$, there exist a morphism $w:P\to Q$ in $\sW(i)$ and a commutative diagram
\begin{equation*}
\xymatrix@C=18pt@R=18pt{
K\ar[d]\ar[r]&P\ar[d]\ar[r]&M\ar[d]\\
L\ar[r]&Q\ar[r]&N.
}
\end{equation*}
It thus suffices to find, for every square of the type on the left, an element of $(\sW\cap\cof I)_c$ factoring it.

For every $i$ and $w$ as above, and every commutative square
\begin{equation*}
\xymatrix@C=18pt@R=18pt{
K\ar[d]_i\ar[r]&P\ar[d]\\
L\ar[r]&Q,
}
\end{equation*}
use the small object argument to factor the morphism $L\coprod_K P\to Q$ through an object $R$ as an element of $\cell I$ followed by an element of $\inj I$. This yields a commutative diagram
\begin{equation*}
\xymatrix@C=18pt@R=18pt{
K\ar[d]\ar[r]&P\ar[d]\ar@{=}[r]&P\ar[d]\\
L\ar[r]&R\ar[r]&Q
}
\end{equation*}
factoring the original square. Furthermore, $P\to R$ is in $\sW$ because $P\to Q$ is in $\sW$ and $R\to Q$ is in $\inj I$, which is assumed to be in $\sW$. Finally, $P\to R$ is in $\cof I$ because it's the composite $P\to L\coprod_K P \to R$ where the first morphism is a pushout of $i$ (hence is a cofibration) and the second is in $\cell I$ (hence is a cofibration). Thus, $P\to R$ is in $\sW \cap \cof I$.

Here we are using the fact that the cofibrations, $\cof I$, are closed under transfinite composition and pushout without any hypothesis on the domains and codomains of the morphisms in question.
\end{proof}

Just as in \cite[Corollary 2.7]{barwickSemi}, we also have a corollary which replaces the set $J$ produced above by a set of morphisms with cofibrant domains.

\begin{corollary} \label{cor:construction-of-J-cof-domains}
Under the conditions of Theorem \ref{thm:smith-thm-semi-output}, a set $J$ can be constructed satisfying the hypotheses of Lemma \ref{lemma:cof-J-characterization} and consisting of morphisms between cofibrant objects.
\end{corollary}

\begin{proof}
Let $J_0$ be the set of morphisms produced by Lemma \ref{lemma:construction-of-J} above. Following \cite[Corollary 2.7]{barwickSemi}, we factorize any commutative square
\begin{equation*}
\xymatrix@C=18pt@R=18pt{
K\ar[d]_i \ar[r]&M\ar[d]^j\\
L\ar[r]&N
}
\end{equation*}

with $i\in I$ and $j\in J_0$ into a commutative diagram
\begin{equation*}
\xymatrix@C=18pt@R=18pt{
K\ar[d]\ar[r]&M' \ar[d]\ar[r]&M\ar[d]\\
L\ar[r]&N'\ar[r]&N
}
\end{equation*}

in which $M'$ is cofibrant and $M'\to N'$ is in $(\sW \cap \cof I)_c$. To do so, use the small object argument to factor $K\to M$ as a cofibration $K\to M'$ followed by a trivial fibration $M'\to M$ and then factor $L\coprod_K M'\to N$ as a cofibration $L\coprod_K M'\to N'$ followed by a trivial fibration $N'\to N$. The morphism $M'\to L\coprod_K M'$ is a pushout of $K\to L$ and so is a cofibration. The morphism $L\coprod_K M'\to N'$ is constructed to be a cofibration. Furthermore, because $M'\to M$, $N'\to N$, and $M\to N$ are weak equivalences the two out of three property implies $M'\to N'$ is a weak equivalence. That $M'$ and $N'$ are cofibrant follows from the fact that $K$ and $L$ are cofibrant, which is part of our hypotheses on $I$. The set of morphisms $M' \to N'$ is the set required.
\end{proof}

Using these lemmas we may prove the theorem.

\begin{proof}[Proof of Theorem \ref{thm:smith-thm-semi-output}]

We check the axioms in Definition \ref{defn:semi} directly. First, observe that $\M$ is assumed to be locally presentable so it is certainly bicomplete. M1 is automatic, as we have previously remarked. M2 is hypothesis (1) of the theorem. For M3, the closure of $\sW$ under retracts is hypothesis (2) of the theorem. Closure of fibrations under retracts follows from the fact that fibrations are defined to be $\inj J$. 
Closure of cofibrations under retracts follows from the fact that the cofibrations are defined to be $\cof I$. This also covers M4i. For M5i, factor a morphism $f$ into an element $i$ of $\cell I$ followed by an element $p$ of $\inj I$. By construction, $p$ is a trivial fibration. Because transfinite composites of pushouts of cofibrations are cofibrations, $i$ is a cofibration.

We turn now to the places where the definition of a semi-model category differs from that of a model category. For M5ii, we must show that any morphism $f:X\to Y$ with a cofibrant domain admits a factorization into a trivial cofibration (i.e. an element of $\sW\cap \cof I$) followed by a fibration. The set $J \subset (\sW \cap \cof I)_c$ produced by Corollary \ref{cor:construction-of-J-cof-domains} has the property that $(\cof J)_c  \supset (\sW \cap \cof I)_c$. With this set in hand we may factor $f$ into $\gamma(f)\circ \delta(f)$ where $\delta(f)$ is in $\cell J$ and $\gamma(f)$ is in $\inj J$ (equivalently, $\gamma(f)$ is a fibration).

If $X$ is cofibrant then the proof of Lemma \ref{lemma:cof-J-characterization} demonstrates that $\delta(f)$ is a trivial cofibration, because the factoring objects $Q^\beta$ are constructed via a transfinite process beginning with $X$ and progressing via pushouts with respect to coproducts of the morphisms in $J$. As each morphism in $J$ is a cofibration between cofibrant objects, these coproducts are again trivial cofibrations between cofibrant objects, and so each pushout is again a morphism of this type. Thus, hypothesis (4) guarantees us that $\delta(f)$ is a trivial cofibration.

For M4ii, let $f$ be a trivial cofibration whose domain is cofibrant, i.e. $f\in (\sW \cap \cof I)_c$. Lemma \ref{lemma:cof-J-characterization} proves $(\sW \cap \cof I)_c = (\cof J)_c$, so $f$ has the left lifting property with respect to $\inj J$ (i.e. with respect to all fibrations).

\end{proof}

\section{Left Bousfield localization for locally presentable semi-model categories} \label{sec:bousfield}

In this section we will use Theorem \ref{thm:smith-thm-semi-output} to prove existence of left Bousfield localization for semi-model categories. We first need a few facts about locally presentable semi-model categories, following \cite{barwickSemi}. The first is the semi-model category analogue of \cite[Proposition 2.5]{barwickSemi}:

\begin{proposition} \label{prop:semi-combinatorial-facts}
Suppose $\M$ is a locally presentable cofibrantly generated semi-model category, with generating cofibrations $I$. For any sufficiently large regular cardinal $\kappa$:
\begin{enumerate}
\item There is a $\kappa$-accessible functorial factorization of each morphism into a cofibration followed by a trivial fibration.
\item There is a $\kappa$-accessible functorial factorization of each morphism with cofibrant domain into a trivial cofibration followed by a fibration.
\item There is a $\kappa$-accessible cofibrant replacement functor.
\item There is a $\kappa$-accessible fibrant replacement functor on cofibrant objects.
\item Arbitrary $\kappa$-filtered colimits preserve weak equivalences with cofibrant domains.
\item Arbitrary $\kappa$-filtered colimits of weak equivalences with cofibrant domains are homotopy colimits.
\item The class $\sW_c$ of weak equivalences with cofibrant domains, is $\kappa$-accessible.
\end{enumerate}
\end{proposition}

\begin{proof}
The proof proceeds just as in \cite[Proposition 2.5]{barwickSemi}, using the small object argument to prove (1)-(4). Note that any object can be fibrantly replaced using the functor of (4), if one first applies cofibrant replacement. To prove (5) and (6), we follow \cite[Proposition 2.5]{barwickSemi}. Given a weak equivalence with cofibrant domain, we factor it into a trivial cofibration followed by a fibration $p$. To prove $p$ is a trivial fibration, we apply a lifting argument against morphisms in $I$, just as in \cite[Proposition 2.5]{barwickSemi}, relying on the $\kappa$-presentability of the domains and codomains of objects in $I$.

To prove (7), we use the $\kappa$-accessible factorization from (2). By the two out of three property, the class $\sW_c$ is the preimage of the class of trivial fibrations under this functor. Once $\kappa$ is chosen large enough that the (co)domains of $I$ are $\kappa$-presentable, the proof that the class of trivial fibrations is accessible follows precisely as it does in \cite[Proposition 2.5]{barwickSemi} and \cite[Corollary A.2.6.6]{lurie-htt}. That is, we realize the class of trivial fibrations as the preimage, under an accessible functor, of an accessible class of morphisms of sets (namely, the surjective morphisms). It follows that $\sW_c$ is $\kappa$-accessible.
\end{proof}

We turn now to left Bousfield localization. The following two theorems combine to prove Theorem \ref{main}. We remind the reader that cofibrantly generated semi-model categories have simplicial mapping spaces defined via hammocks \cite[Notation 3.61]{barwickSemi}, that we will denote $map(-,-) \in sSet$. Such mapping spaces may be computed via cosimplicial and simplicial resolutions, up to a zig-zag of weak equivalences \cite[Scholium 3.64]{barwickSemi}.

Given a class of morphisms $\cat C$ in a cofibrantly generated semi-model category $\M$, an object $W$ is called \textit{$\cat C$-local} if it is fibrant in $\M$ and $map(f,W)$ is a weak equivalence of simplicial sets for all $f\in \cat C$. A morphism $g$ in $\M$ is a \textit{$\cat C$-local equivalence} if $map(g,W)$ is a weak equivalence for all $\cat C$-local objects $W$. Several properties about $\cat C$-local objects and equivalences, proven in \cite{hirschhorn} without reference to a model structure on $\M$, will be used below. Because a set of morphisms $\cat C$ can always be replaced by a set of cofibrations between cofibrant objects, without changing the left Bousfield localization $L_{\cat C}\M$, we will always assume $\cat C$ is a set of cofibrations between cofibrant objects. For the proof that follows, we advise the reader to have copies of \cite{barwickSemi, hirschhorn} on hand. 

\begin{theorem} \label{thm:loc-exists}
If $\M$ is a locally presentable, cofibrantly generated semi-model category in which the domains of the generating cofibrations are cofibrant. For any set of morphisms $\cat C$ in $\M$, there exists a cofibrantly generated semi-model structure $L_{\cat C}(\M)$ on $\M$ with weak equivalences defined to be the $\cat C$-local equivalences, (generating) cofibrations defined to match the (generating) cofibrations of $\M$, and fibrations defined by the right lifting property with respect to some set $J_{\cat C}$ of $\cat C$-local equivalences which are also cofibrations with cofibrant domains. Furthermore, the fibrant objects of $L_{\cat C}\M$ are precisely $\cat C$-local fibrant objects in $\M$.
\end{theorem}

\begin{proof}
The semi-model structure on $L_{\cat C}(\M)$ will be obtained via Theorem \ref{thm:smith-thm-semi-output} as soon as we check conditions (1)-(5). We begin with condition (1). First, Lemma 4.5 of \cite{barwickSemi} states that the class of $\cat C$-local objects is an accessibly embedded, accessible subcategory of $\M$. This lemma remains true for semi-model categories, since the proof only requires the existence of an accessible fibrant replacement functor on cofibrant objects (thanks to the derived hom) and the fact that the subcategory of weak equivalences of simplicial sets is accessibly embedded and accessible. Next, Lemma 4.6 in \cite{barwickSemi} implies that the class of $\cat C$-local equivalences with cofibrant domains is an accessibly embedded, accessible subcategory of $Arr(\M)$. Barwick's proof only requires that $\kappa$-filtered colimits are homotopy colimits for sufficiently large $\kappa$, and again this holds for semi-model categories (see Proposition \ref{prop:semi-combinatorial-facts}).

This completes the verification of (1), but for the reader's convenience, we will explain how to directly prove that the collection $\cat W_{\cat C,c}$ of $\cat C$-local equivalences with cofibrant domains is an accessibly embedded, accessible subcategory of $Arr(\M)$, by realizing $\cat W_{\cat C,c}$ as the inverse image of the collection of weak equivalences of $Arr(sSet)$ under an accessible functor. Thanks to \cite[Lemma 4.5]{barwickSemi}, we can choose $\lambda$ such that $\M$ is $\lambda$-presentable and the collection of $\cat C$-local objects is a $\lambda$-accessible subcategory (hence is generated under $\lambda$-filtered colimits by a set $\cat L$ of objects). Define a functor $F: Arr(\M) \to Arr(sSet)$ by $F(f) = \coprod_{Z\in \cat L} map(f,Z)$; it is contravariant accessible, since $map(\colim f_\alpha,Z)\cong \lim map(f_\alpha,Z)$, so $F$ preserves colimits when viewed as a covariant functor $F: Arr(\M)\to Arr(sSet)^{op}$. By construction, $F$ takes $\cat W_{\cat C,c}$ to weak equivalences of simplicial sets. Next, we show $\cat W_{\cat C,c} = F^{-1}(\cat W_{sSet})$. If $f:A\to B$ is a morphism such that $F(f)$ is a weak equivalence, and if $W$ is a $\cat C$-local object, then $W$ can be written as a $\lambda$-filtered colimit $W\cong \colim Z_\alpha$ of objects in $\cat L$. Using that $A$ and $B$ are $\lambda$-presentable, $map(f,W)\cong map(f,\colim Z_\alpha) \cong \colim map(f,Z_\alpha)$ is a $\lambda$-filtered colimit of weak equivalences of simplicial sets and hence is a weak equivalence. So $f$ is a $\cat C$-local equivalence as required.

The closure of the class of $\cat C$-local equivalences under the two out of three property is proven in a similar way to Proposition 3.2.3 in \cite{hirschhorn}. Namely, given $g,h,h\circ g$ one applies functorial cofibrant replacement. Given a $\cat C$-local object $W$, one applies the simplicial mapping space functor $map(-,W)$ to the diagram
\[
\xymatrix{\tilde{X} \ar[r]^{\tilde{g}} \ar[d] & \tilde{Y} \ar[r]^{\tilde{h}} \ar[d] & \tilde{Z} \ar[d]  \\  
X \ar[r]_g & Y \ar[r]_h & Z}
\]
and one then applies the two out of three property in sSet.

The closure of the class of $\cat C$-local equivalences under retract is proven analogously, following Proposition 3.2.4 in \cite{hirschhorn}, which again applies cofibrant replacement to the morphisms in question and then considers the morphisms induced in $sSet$ by $map(\tilde{f},W)$ for all $\cat C$-local $W$. This completes our proof of (2).

For (3), note that if $f$ is in inj$(I)$ in $L_{\cat C}(\M)$ then $f$ is in inj$(I)$ in $\M$. Thus, $f$ is a trivial fibration in $\M$. So all we need to show is that a weak equivalence of $\M$ is a $\cat C$-local equivalence. This is true by general properties of simplicial mapping spaces, even in a semi-model category \cite[Corollary 3.66]{barwickSemi}.

We must now check (4). Suppose $f:A\to B$ is an element of $(\sW \cap \cof I)_c$, i.e., a $\cat C$-local equivalence and a cofibration between cofibrant objects. Suppose $g:A\to X$ is any morphism such that $X$ is cofibrant. Then the pushout
\begin{align*}
\xymatrix{A \ar[r] \ar[d] \po & B\ar[d] \\ X\ar[r] & P}
\end{align*}
has $h:X\to P$ a cofibration and we must show it's a $\cat C$-local equivalence. We note that this square is a homotopy pushout square, because one leg is a cofibration, and all objects are cofibrant \cite[page 10]{spitzweck-thesis}. Note that this is where left properness would normally be required, but we don't need it because we assume $X$ is cofibrant. We fix a $\cat C$-local object $W$ and apply $map(-,W)$. 
Following \cite[Theorem 4.7]{barwickSemi}, note that the following is a homotopy pullback diagram in $sSet$:
\begin{align*}
\xymatrix{map(P,W) \ar[r] \ar[d] & map(X,W) \ar[d] \\ map(B,W) \ar[r] & map(A,W).}
\end{align*}
Since $f$ is a $\cat C$-local equivalence, the bottom horizontal map is a weak equivalence, hence so is the top map, hence $h$ is a $\cat C$-local equivalence as required. Next, consider a transfinite composition of elements of $(\sW \cap \cof I)_c$, encoded by a functor $F:\lambda \to \M$ and denoted $X_0 \to X_1 \to \dots$. The composition $cF: X_0 \to \colim X_\alpha$ is an element of $\cof I$, so we must show it's also a $\cat C$-local equivalence. For left proper model categories, this follows from \cite[Proposition 3.2.11]{hirschhorn}. For semi-model categories, it is even easier, because the colimit is already a homotopy colimit, since the diagram is Reedy cofibrant. We prove $cF$ is a $\cat C$-local equivalence by transfinite induction, noting that the case of finite compositions and successor ordinals follows from the two out of three property. So, let $\lambda$ be a limit ordinal and consider the following ladder, where the vertical morphisms are compositions:

\[
\xymatrix{
X_0 \ar[d] \ar@{=}[r] & X_0 \ar[d]\ar@{=}[r] & X_0\ar[d] \ar@{=}[r] & \dots \ar@{=}[r] & X_0 \ar@{..>}[d]\\
X_1 \ar[r] & X_2 \ar[r] & X_3 \ar[r] & \dots \ar[r] & \colim X_\alpha}
\]

All objects are cofibrant and all horizontal morphisms are in $(\sW \cap \cof I)_c$. The solid vertical morphisms are all $\cat C$-local equivalences by the inductive hypothesis. Hence, the dotted arrow, $cF$, is a $\cat C$-local equivalence as well, by applying $map(-,Z)$ for a $\cat C$-local object $Z$ to convert this into a homotopy limit diagram in the category of simplicial sets.

Finally, condition (5) is part of the hypotheses, since we assume $\M$ is a cofibrantly generated semi-model category with domains of $I$ cofibrant. We now prove the last sentence of the statement of the theorem. If $W$ is fibrant in $L_{\cat C}\M$ then by \cite[Corollary 3.66]{barwickSemi}, $W$ is $\cat C$-local. It is also clear that $W$ is fibrant in $\M$, e.g., because $id:\M \leftrightarrows L_{\cat C} \M:id$ is a Quillen pair. If $W$ is $\cat C$-local and fibrant in $\M$, one can prove it is fibrant in $L_{\cat C} \M$ by following Hirschhorn's theory of homotopy orthogonal pairs \cite[Propositions 17.8.5, 17.8.9]{hirschhorn}, which makes use of properties of the model category $sSet$, and holds for semi-model categories. Since $L_{\cat C} \M$ is a \textit{cofibrantly generated} semi-model category, $W\to \ast$ is a fibration if and only if it satisfies the right lifting property with respect to generating trivial cofibrations $i: A\to B$, which are all morphisms between cofibrant objects. Hence, Hirschhorn's proof directly translates to the semi-model category context, cosimplicially replacing $i$ to compute simplicial mapping spaces and using the theory of homotopy orthogonal squares to construct a lift proving that $W$ is fibrant in $L_{\cat C}\M$.
\end{proof}

We turn now to verifying the universal property of left Bousfield localization (with respect to left Quillen functors of semi-model categories \cite[Definition 1.12]{barwickSemi}).

\begin{theorem} \label{thm:universal}
Suppose that there is a cofibrantly generated semi-model structure $L_{\cat C}(\M)$ on the semi-model category $\M$ as in Theorem \ref{thm:loc-exists}. Then $L_{\cat C}(\M)$ satisfies the following universal property. Suppose $F:\M\to \cat{N}$ is any left Quillen functor of semi-model categories taking $\cat C$ into the weak equivalences of $N$. Then $F$ is a left Quillen functor when viewed as $F:L_{\cat C}(\M)\to \cat{N}$.
\end{theorem}

To prove this, we need a lemma, inspired by \cite[Proposition 8.5.3]{hirschhorn}.

\begin{lemma} \label{lemma:right-Quillen}
Let $F:\cat M \leftrightarrows \cat N: U$ be a pair of adjoint functors between semi-model categories $\cat M$ and $\cat N$, and assume $\M$ is cofibrantly generated with a set of generating trivial cofibrations $J$ with cofibrant domains. Then the following are equivalent: 

\begin{enumerate}
\item $(F,U)$ is a Quillen pair.
\item $F$ preserves cofibrations and preserves trivial cofibrations between cofibrant objects.
\item $U$ preserves fibrations and trivial fibrations.
\item $F$ preserves cofibrations and $U$ preserves fibrations.
\item $F$ preserves trivial cofibrations whose domain is cofibrant and $U$ preserves trivial fibrations.
\end{enumerate}
\end{lemma}

\begin{proof}
The equivalence of (1) and (3) is part of the definition of a Quillen pair for semi-model categories \cite[Definition 1.12]{barwickSemi}, and (3) implies (2) is proven in \cite[Section 12.1.8]{fresse-book}. For (2) implies (3) we use an adjoint lifting argument and the assumptions on $\M$. For (4), we use that the hypothesis on $F$ implies $U$ preserves trivial fibrations, since trivial fibrations are characterized as morphisms satisfying the right lifting property with respect to cofibrations \cite[Lemma 1.7.1]{barwickSemi}. For (5), we use that the hypothesis on $U$ implies $F$ preserves cofibrations, which are characterized as morphisms satisfying the left lifting property with respect to trivial fibrations \cite[Lemma 1.7.1]{barwickSemi}.
\end{proof}

\begin{proof}[Proof of Theorem \ref{thm:universal}]
Let $G:\cat{N}\to \M$ be the right adjoint of $F$, and let $U$ denote $G$ viewed as a functor from $\cat{N}$ to $L_{\cat C}(\M)$, since as categories $\M$ and $L_{\cat C}(\M)$ are equal. We must prove $U$ is right Quillen \cite[Definition 1.12]{barwickSemi}. The trivial fibrations of $L_{\cat C}(\M)$ are equal to the trivial fibrations of $\M$, since both are characterized as $\inj I$. Thus, $U$ preserves trivial fibrations. We will now prove $F$ preserves trivial cofibrations whose domain is cofibrant, which is sufficient by Lemma \ref{lemma:right-Quillen}.

Let $g$ be a trivial cofibration between cofibrant objects in $L_{\cat C}\M$. We already know that $Fg$ is a cofibration, since the cofibrations of $\M$ and $L_{\cat C}\M$ coincide. To prove that $Fg$ is a weak equivalence in $\cat N$, it suffices to prove, for every fibrant $X$ in $\cat N$, that $map(Fg,X)$ is a weak equivalence of simplicial sets. Using \cite[Scholium 3.64]{barwickSemi}, we see that $map(Fg,X) \simeq map(g,UX)$. It is therefore sufficient to prove that $UX$ is a $\cat C$-local object, i.e., that $U$ takes fibrant objects of $\cat N$ to local objects of $\M$. 

To prove this, let $f$ be a morphism in $\cat C$, and note that, by \cite[Scholium 3.64]{barwickSemi} again, $map(f,UX) \simeq map(Ff,X)$. Since $\cat C$ consists of cofibrations between cofibrant objects, and $Ff$ is a weak equivalence by assumption, we see that $map(Ff,X)$ is a weak equivalence of simplicial sets, proving that $UX$ is $\cat C$-local as required.

\end{proof}

\begin{remark}
While we have not needed further results from \cite{hirschhorn}, effectively every result in \cite{hirschhorn} has an analogue for semi-model categories, sometimes with additional cofibrancy hypotheses. As part of a longer proof of Theorem \ref{thm:universal}, we proved semi-model categorical analogues of \cite[Proposition 7.2.18]{hirschhorn} (where in (2), domains and codomains must be cofibrant), \cite[Proposition 8.5.4]{hirschhorn} (where in (3) objects must be bifibrant instead of just fibrant), and \cite[Theorem 17.7.7]{hirschhorn} (about Reedy cofibrant replacements and simplicial mapping spaces). A key point was that several other authors had already worked out that Reedy semi-model categories behave precisely like Reedy model categories \cite{barwickSemi, goerss-hopkins, spitzweck-thesis}. We also verified several useful facts about the semi-model category $L_{\cat C} \M$, including \cite[Proposition 3.3.16]{hirschhorn} (characterization of local fibrations between local objects, if the domain is cofibrant), \cite[Theorem 3.2.13]{hirschhorn} (characterization of local equivalences between bifibrant objects in $L_{\cat C}\M$), and \cite[Theorem 3.1.6]{hirschhorn} (subsumed by Theorem \ref{thm:universal} above). 
\end{remark}

We conclude this section with an application of Proposition \ref{prop:semi-combinatorial-facts} to prove a semi-model categorical version of a result of Dugger \cite[Corollary 1.2]{dugger}. This has the pleasant property of demonstrating that the theory of combinatorial semi-model categories is homotopically the same as the theory of combinatorial model categories (and, hence, of presentable $\infty$-categories). This result is not required for the rest of the paper, but is an important property of the theory of combinatorial semi-model categories as a whole, and answers a question Zhen Lin Low once asked the second author.

\begin{proposition} \label{prop:dugger}
Every combinatorial semi-model category $\M$ is Quillen equivalent, as a semi-model category, to a left proper, combinatorial model category where all objects are cofibrant.
\end{proposition} 

\begin{proof}
Let $\lambda$ be the cardinal for which $\M$ is $\lambda$-locally presentable, and let $\M_\lambda$ denote a dense subcategory of $\M$ (such that every object of $\M$ is isomorphic to its canonical colimit with respect to $\M_\lambda$). Following \cite[Section 3]{dugger}, we first consider the case where $\M$ is a simplicial semi-model category (defined in \cite{goerss-hopkins}). We can produce a small category $\cat C$ and a left proper, simplicial, combinatorial model category $U\cat C := sSet^{\cat C^{op}}$ (the projective model structure) and a set of morphisms $S$ such that $L_S U\cat C$ is Quillen equivalent to $\M$. Use of the injective model structure on $sSet^{\cat C^{op}}$ provides the model where all objects are cofibrant. 

The key point is that finding a homotopically surjective map $U\cat C \to \M$ is equivalent to finding a function $\gamma: \cat C \to \M$ such that for every fibrant $X$ in $\M$, $X$ is weakly equivalent to the `canonical homotopy colimit' $\hocolim(\cat C \times \Delta \downarrow X)$. The category $\cat C$ is produced in \cite{dugger} as $\M_\lambda^{cof}$, the subcategory of cofibrant objects in $\M_\lambda$. Dugger's proof that this $\cat C$ has the required property boils down to the existence and homotopy invariance of cosimplicial resolutions (proven for semi-model categories in \cite{goerss-hopkins}) and the properties listed in Proposition \ref{prop:semi-combinatorial-facts}; see \cite[Proposition 4.7]{dugger}. It is important to note that none of Dugger's proofs require the model category axioms where model categories and semi-model categories differ. Dugger's results about $U\cat C$ work verbatim, whether $\M$ is a model category of semi-model category, e.g., Propositions 3.2 (noting that $LA^{cof}$ is cofibrant in the semi-model category $\cat N$), 4.2, and 4.6 (which only makes reference to the subcategory of cofibrant objects in $\M$).

For the case where $\M$ is not simplicial, we follow the proof in \cite[Section 6]{dugger}, replacing $\M$ by the category of cosimplicial objects $c\M$ with its Reedy semi-model structure \cite{barwickSemi, goerss-hopkins}. Dugger works in the subcategory $\cat C R$ of $c\M$ consisting of cosimplicial resolutions $A^*$ where $A^n \in \M_\lambda^{cof}$ for all $n$. Since everything in sight is cofibrant, Dugger's arguments work verbatim when $\M$ is only a semi-model category.
\end{proof}

\begin{remark}
Dugger's theorem is not the only way to replace a combinatorial semi-model category by a Quillen equivalent model category. An alternative way is to generalize \cite[Theorem 2.4]{ching-riehl} to semi-model categories. The new version states that, if $\M$ is a combinatorial and simplicial semi-model category, then there exists a simplicially-enriched cofibrant replacement comonad $c: \M \to \M$ that preserves $\lambda$-filtered colimits. Furthermore, the category $\M_c$ of $c$-coalgebras inherits a model structure from $\M$, transferred along the forgetful functor $U: \M_c\to \M$, where a morphism $f$ is a weak equivalence (resp. cofibration) if $U(f)$ is in $\M$. This is a full model structure because every object in $\M_c$ is cofibrant. This result can be proved by carefully reading \cite{ching-riehl} and noting that their categorical arguments apply directly to combinatorial semi-model categories (including the use of Garner's small object argument), and their homotopical arguments only ever require the model category axioms on cofibrations (rather than trivial cofibrations), hence work in semi-model categories. In particular, \cite[Lemma 1.6]{ching-riehl} works with the (cofibration, trivial fibration) factorization, and \cite[Theorem 2.4]{ching-riehl} uses the fact that cofibrations lift against trivial fibrations.  We omit the details because this plan has already been carried out, in fact in more generality (going from a combinatorial weak model category to a combinatorial model category) in \cite[Proposition 5.1]{bourke-henry}.
\end{remark}

\section{Applications} \label{sec:applications}

In this section we provide numerous applications of Theorem \ref{main}. Most of these applications are model categories that fail to be left proper. We begin with an important example of Voevodsky \cite{voevodsky} which we mentioned in the Introduction. We then discuss our main application, to our companion paper \cite{Reedy-paper}. Finally, we explore applications to categories of algebras over operads, spectra/stabilization, (weakly) enriched categories, and Goodwillie calculus.

\subsection{Radditive functors}
\begin{example} \label{ex:voevodsky}
In \cite{voevodsky}, Voevodsky introduced a model categorical framework for the study of simplicial extensions of functors. He defined a functor $F: \cat{C}^{op} \to Set$ to be radditive if $F(\emptyset) = pt$ and $F(X\coprod Y) \cong F(X)\times F(Y)$. He introduced a combinatorial model structure on the category of simplicial objects in radditive functors, but needed to assume it was left proper (e.g., in \cite[Theorem 3.46]{voevodsky}), in order to localize it. Indeed, \cite[Proposition 3.35]{voevodsky} characterizes when this model structure is left proper, and \cite[Example 3.48]{voevodsky} is a case where left properness fails and left Bousfield localization (as a model category) provably does not exist. The failure is that a certain pushout of a trivial cofibration is not a weak equivalence. 

However, this obstruction does not prevent the existence of a semi-model structure, and indeed, Theorem \ref{main} provides a semi-model structure for Voevodsky's example. Voevodsky writes his paper carefully, to prove results about the local homotopy category, even when left properness fails. With Theorem \ref{main}, the local homotopy category may be studied via the local semi-model structure, providing many tools beyond those available to Voevodsky (e.g., fibrant replacement for the computation of homotopy limits of diagrams of simplicial radditive functors).
\end{example} 

\subsection{Inverting Unary Operations}

The classical theory of localization of categories is concerned with inverting of a set of morphisms in a small category in a universal way. The resulting localized category can often  be studied through its category of presheaves. In the world of operads, we can also try to localize operads by inverting a specified set of unary operations, and then study algebras over these localized operads. In homotopy theory it is natural to require a form of weak invertibility. This means that we want to study localized operads and their algebras where the operations from a specified set of unary morphisms act as weak equivalences on the level of algebras. This is the main subject of our companion paper \cite{Reedy-paper}. We briefly describe the results from this paper here and refer the reader to \cite{Reedy-paper} for the details. 

\begin{example} The categories of presheaves with values in a model category $\M$ is a particularly simple case of the situation described above. Cisinski studied localizations of the covariant presheaf categories $[\calc,\M]$ in \cite{Cis06,cisinski-locally-constant} when $\M$ is a left proper combinatorial model category and $[\calc,\M]$ is equipped with the projective or injective model structure. The resulting localized category $[\calc,\M]^{loc}$ has as fibrant objects presheaves for which each morphism in $\calc$ acts as a weak equivalence in $\M.$ 
 
In \cite{Reedy-paper} Using Theorem \ref{main} we extend the results of Cisinski to an arbitrary combinatorial model category $\M$ and, moreover, we consider the semi-model Bousfield localization $[\calc,\M]^{W}$ whose fibrant objects are locally constant presheaves with respect to an arbitrary proper Grothendieck fundamental localizer $W$ and an arbitrary subset of morphisms of $\calc$. The case studied by Cisinski corresponds to the minimal fundamental localizer $W=W_{\infty}.$
 \end{example} 
 
\begin{example} \label{ex:localisation of algebras} 
In \cite{Reedy-paper}, we extend the example above, and localize the category of algebras of a $\Sigma$-free colored operad $P$ by lifting Cisinski's localizations $[\calc,\M]^{W}$ to the category of algebras of $P$ with values in a combinatorial symmetric monoidal model category $\M.$ This is where we need the full power of our Theorem \ref{main} because even if $\M$ is left proper, the projective model structure on the category of algebras of $P$ is most often not a left proper category  \cite{hackney} (for this category to be left proper we need $\M$ to be strongly $h$-monoidal and $P$ be tame, which is a rare occasion \cite{batanin-berger}).    
 \end{example} 
 
 \begin{example} \label{ex:main-application}
 Our main application, contained in \cite{Reedy-paper}, proves a strong form of the Baez-Dolan stabilization hypothesis \cite{baez-dolan}, using $k$-operads valued in weak $n$-categories \cite{Rezk} to model $k$-tuply monoidal weak $n$-categories. 
  We apply Theorem \ref{main} to construct a semi-model categorical left Bousfield localization of the category of $k$-operads, with respect to the  Grothendieck fundamental localizer of $n$-homotopy types, $W_n$, as we now describe.

 Let $\M$ be a combinatorial monoidal model category with cofibrant unit.
 The category $Op_k(\M)$ of $k$-operads in $\M$ is encoded as the category of algebras of a $\Sigma$-cofibrant colored operad whose underlying category is the opposite of the category of quasibijections of $k$-ordinals $Q_k^{op}.$ Hence, $Op_k(\M)$ has a semi-model  structure transferred from the projective model category structure on the category $[Q_k^{op},\M]$, as we prove in \cite{Reedy-paper}. Following Example \ref{ex:localisation of algebras} we now construct the localization of the category of $k$-operads $Op^{W_n}_k(\M)$ whose fibrant objects are $W_n$-locally constant $k$-operads i.e. $k$-operads whose underlying presheaves on $Q_k^{op}$ are $W_n$-locally constant. This lifts \cite[Theorems 7.1 and 7.2]{batanin-locally-constant} from the homotopy category level to the semi-model category level.
 
 We then prove the following Stabilization Theorem for $k$-operads: if $k\ge n+2$ the natural pair of adjoint functors (suspension and restriction) between  $Op^{W_n}_k(\M)$ and $Op^{W_n}_{k+i}(\M)$ is a Quillen equivalence for any $1\le i \le \infty.$ This is our strong form of the Baez-Dolan stabilization hypothesis. The original Baez-Dolan stabilization for $k$-tuply monoidal weak $n$-categories follows from this Theorem if we take  as $\M$ to be Rezk's model category of $\Theta_n$-spaces \cite{Rezk} and consider the value of the left derived  suspension functor on a contractible $k$-operad (see  \cite{batanin-baez-dolan-via-semi} for an explanation how to model  $k$-tuply monoidal weak $n$-categories as algebras of $k$-operads).
  
\end{example}

Other approaches to the problem of weakly inverting of unary operations in operads were recently proposed. We briefly describe them below and indicate where our results can be used for further improvement.

\begin{example} \label{ex:tillmann}
Motivated by topological and conformal field theories, two recent papers \cite{tillmann1, tillmann2} provide an analogue for operads of the Dwyer-Kan hammock localization of categories, that allows one to weakly invert some unary operations in an operad. Algebras of such `localized' operads can be interpreted as algebras of the original operad where some set of unary operations are weak homotopy equivalences (see section 6 of \cite{tillmann1}). This localization can be studied with Theorem \ref{main}. 
\end{example}

\begin{example} \label{ex:lazarev}
In \cite{lazarev}, localizations at the level of $R$-modules are compared to localizations of dg-algebras, where $R$ is a dg-ring. As observed in \cite[Remark 2.13]{lazarev}, the category of dg-$R$-algebras is not left proper in general (it is if $R$ is a field), and thus cofibrant replacements of dg-algebras are often required in \cite{lazarev}. Theorem \ref{main} can be used to streamline the exposition of \cite{lazarev}, by providing localizations even when left properness fails. Recalling that operads are monoids with respect to the circle product \cite{white-yau}, we could also use Theorem \ref{main} to extend the ideas in \cite{lazarev} to the study of localization of operations in categories of operads, algebras, and modules.
\end{example}

\subsection{Localizing categories of algebras over colored operads}

One of the crucial ideas in Example \ref{ex:main-application} is that a localization of a category of algebras has especially nice properties if it coincides with an appropriate transferred semi-model structure, as explored in \cite{batanin-white-eilenberg-moore}. One of the main technical achievements of our paper \cite{Reedy-paper} is the set of nontrivial combinatorial conditions on the operad $P$ when such a coincidence does occur.

In fact, this kind of situation, of wanting to localize a transferred (semi-)model structure, is ubiquitous. In the following examples, we always use $I$ to denote the generating cofibrations, and we recall that a model category is called \textit{tractable} if it is combinatorial and has the domains of the generating (trivial) cofibrations cofibrant \cite{barwickSemi}. The following lemma is often useful, and requires slightly less than tractability:

\begin{lemma} \label{lemma:tractable-transfer}
Suppose $\M$ is a combinatorial model category with domains of the generating (trivial) cofibrations cofibrant. Suppose $F: \M \leftrightarrows \cat N: U$ is an adjoint pair such that the monad $T = U \circ F$ is accessible (preserves $\lambda$-directed colimits). Suppose $\cat N$ admits a transferred model structure from $\M$, i.e., a morphism in $\cat N$ is a weak equivalence (resp. fibration) if and only if $U(f)$ is in $\M$. Then $\cat N$ is combinatorial and has domains of the generating (trivial) cofibrations cofibrant.
\end{lemma}

\begin{proof}
First, $\cat N$ is locally presentable because $T$ is accessible, by \cite[2.47, 2.78]{adamek}. It is standard (see, e.g., \cite{white-yau}) that the generating (trivial) cofibrations are of the form $F(I)$ (resp. $F(J)$) where $I$ (resp. $J$) are the generating (trivial) cofibrations of $\M$. To see that the domains of morphisms in $F(I)$ (resp. $F(J)$) are cofibrant in $\cat N$ is now a simple lifting argument, using the adjunction. 
\end{proof}

Finally, we recall that Jeff Smith's category of $\Delta$-generated spaces is a tractable model category Quillen equivalent to the usual model category of spaces, as has been proven in a preprint of Dan Dugger, and in published work of Philippe Gaucher. Details of this model structure, as well as how to build a tractable model for orthogonal spectra based on $\Delta$-generated spaces, may be found in \cite{white-localization}.

\begin{example} \label{ex:Goerss-Hopkins}
The resolution model structure, also known as the $E_2$-model structure, is described in \cite{goerss-hopkins}. It was introduced by Dwyer, Kan, and Stover in the context of pointed topological spaces, and generalized by Bousfield to the setting of general left proper model categories $\M$. It is a model structure on cosimplicial objects $c\M$, with more weak equivalences than the usual Reedy model structure. The weak equivalences are morphisms that induce an isomorphism on the $E_2$-term of certain spectral sequences (or on the $E^2$-term for simplicial objects). If $\M$ is tractable, then so is the resolution model structure, as shown in \cite[1.4.10]{goerss-hopkins} (see also right after Theorem 1.4.6).

In \cite{goerss-hopkins}, Goerss and Hopkins transfer the resolution model structure on simplicial spectra (any choice of $S$-modules, orthogonal spectra, or symmetric spectra) to simplicial $T$-algebras, for a well-behaved simplicial operad $T$. As the positive (or positive flat) model structure is used on symmetric spectra, the model category of $T$-algebras is tractable by Lemma \ref{lemma:tractable-transfer} (at least, if a combinatorial model is used for spaces as recalled above), but not left proper. For this reason, Goerss and Hopkins developed a semi-model categorical localization technique (using the language of cellular, rather than combinatorial, model categories) to prove the existence of $E_*$-localization, for a generalized homology theory $E$, on simplicial $T$-algebras.  The existence of $E_*$-localization as a semi-model category also follows from Theorem \ref{main}, recovering Theorem 1.5.1 of \cite{goerss-hopkins}.
\end{example}

A similar example arises in the study of $TQ$-localization for categories of algebras over an operad $\cat O$ acting in spectra. If a combinatorial model category of spectra is used, such as the positive model structure on symmetric spectra, then for any $\cat O$, the category of $\cat O$-algebras admits a transferred tractable model structure, $\M$, by Lemma \ref{lemma:tractable-transfer} and \cite{white-yau}, and hence Theorem \ref{main} applies.

\begin{example} \label{ex:harper-zhang}
In order to left Bousfield localize with respect to the class of $TQ$-homology isomorphisms (or $TQ$-homology with coefficients), we must first reduce to a set $\cat C$ of morphisms. This is done in \cite{harper-zhang}. While $\cat M$ is almost never left proper, the semi-model categorical localization $L_{\cat C} \M$ guaranteed by Theorem \ref{main} matches that of \cite[Theorem 5.14]{harper-zhang}, providing an alternative proof of the main result of \cite{harper-zhang}.
\end{example}

We conclude with one more example of localizing categories of algebras over operads in spectra.

\begin{example}\label{ex:beardsley}
In \cite{beardsley}, Beardsley initiates a program of learning about the homotopy groups of Ravenel's $X(n)$-spectra via $E_k$-cell attachments. The first setting of the paper is $E_1$-$X(n)$-algebras, i.e., $E_1$-algebras in the monoidal category of $X(n)$-modules. This category admits a transferred, tractable, left proper model structure (by Lemma \ref{lemma:tractable-transfer} and \cite{batanin-berger}). Beardsley constructs localizations of this model structure with respect to a prime $p$, and his techniques could also be used to localize with respect to $E_*$-equivalences for various generalized cohomology theories $E$, such as $K(n)$. Beardsley next introduces an $E_k$-monoidal analogue of $X(n)$, denoted $X(n,k)$. His work attaching $E_k$-cells most naturally takes place in $E_k$-$A$-algebras (where $A$ is one of the spectra $X(n,k)$), and as categories of $E_k$-algebras are not known in general to be left-proper, Theorem \ref{main} is required to construct the left Bousfield localizations for this new setting, and to prove (following \cite{batanin-white-eilenberg-moore}) that these localizations play nicely with colimits in categories of $E_k$-$A$-algebras.
\end{example}

\subsection{Stabilization and Spectra}

One of the most common applications of left Bousfield localization is to build stable model categories of spectra in some base model category $\M$ \cite{hovey-spectra}. The idea here is to first build spectra $Sp(\M)$ as sequences of objects of $\M$, with levelwise weak equivalences, and then localize with respect to stable equivalences (relative to some endofunctor $G$ on $\M$ that generalizes reduced suspension on pointed spaces). If $\M$ is left proper, then so is the levelwise model structure on $Sp(\M)$. Otherwise, it is not known how to build the stable model structure on $Sp(\M)$. There are many places in the literature where various authors wished to build a stable model structure, but could not because $\M$ was not known to be left proper. We review several such places below, and more in Examples \ref{ex:pereira} and \ref{ex:vicinsky}.

\begin{example} \label{ex:intermont}
In \cite{johnson}, Intermont and Johnson introduced model structures for the category of ex-spaces, suitable for the study of parameterized unstable homotopy theory. Both their coarse model structure (which is left proper \cite[Proposition 3.1]{johnson}) and their $\cat U$-model structure (which is not known to be left proper \cite[Remark 5.6]{johnson}) have several improvements over the model structure used by May and Sigurdsson. However, it is left as an open problem to construct a suitable homotopy theory for parameterized spectra based on the $\cat U$-model structure. With Theorem \ref{main}, this problem can be solved. As the $\cat U$-model structure is obtained as a transfer from $Top$, it will be tractable if a tractable model structure (e.g. $\Delta$-generated spaces) for spaces is chosen, by Lemma \ref{lemma:tractable-transfer}. With that tractable model structure in hand, Hovey's stabilization machinery \cite{hovey-spectra} may be used, resulting in a stable semi-model structure for parameterized spectra based on the $\cat U$-model structure.
\end{example}

\begin{example} \label{ex:monoids-in-spectra}
Let $\M$ be a combinatorial monoidal semi-model category and $K$ a cofibrant object of $\M$. Mark Hovey's stabilization machinery \cite{hovey-spectra} produces the category $Sp^{\Sigma}(\M,K)$ of symmetric spectra valued in $\M$, with respect to the endofunctor $- \otimes K$. The argument of \cite[Theorem 8.2]{hovey-spectra} works without change to produce the projective semi-model structure on $Sp^{\Sigma}(\M,K)$, which is again combinatorial, and is monoidal by the argument of \cite[Theorem 8.3]{hovey-spectra}. Theorem \ref{main} then produces the stable semi-model structure on $Sp^{\Sigma}(\M,K)$, where $-\otimes K$ is a Quillen equivalence. This is a semi-model version of \cite[Theorem 8.11]{hovey-spectra}, but without the need to assume $\M$ is left proper. Just after \cite[Theorem 8.11]{hovey-spectra}, Hovey considers the category of $R$-modules for a given monoid $R$ in $Sp^{\Sigma}(\M,K)$. Even when $\M$ is a left proper model category, Hovey is unable to prove the existence of the stable model structure on $R$-modules or on $R$-algebras, in the latter case writing `This plan will certainly fail for the category of monoids, since the projective model structure on monoids will not be left proper in general.' With our Theorem \ref{main}, we are able to carry out Hovey's plan. Precisely, we apply \cite[Theorem 2.2.1]{Reedy-paper} to produce the transferred projective semi-model structure on $R$-modules or on $R$-algebras, and then we apply Theorem \ref{main} to localize these semi-model structures to obtain stable semi-model structures on each category. The same plan can be carried out for algebras over any $\Sigma$-cofibrant colored operad, and indeed for more general classes of colored operads if $\M$ is made to satisfy more conditions \cite{white-yau} (or \cite{white-commutative-monoids} for commutative monoids). Hovey's alternative idea, of finding conditions so that localization preserves the monoid axiom (so that the stable model structure on $Sp^{\Sigma}(\M,K)$ satisfies the monoid axiom) has also been worked out, in \cite{white-localization}. We do not need the monoid axiom with the present approach, since it is only there to get from a semi-model structure on monoids to a full model structure, but a semi-model structure is sufficient for Theorem \ref{main}.

\end{example}

\begin{example} \label{ex:kasparov}
In \cite[Theorem 9.6]{kk}, Joachim and Johnson introduced a model structure on a particular category of $C^*$-algebras, that can be used in the study of Kasparov's $KK$-theory. This model structure is obtained as a transfer from the category $Top_*$ of pointed topological spaces (but where the left adjoint $F$ is only defined on compact spaces). The generating cofibrations have the form $F(i_k)$ where $i_k: S^{k-1}_+\to D^K_+$, and hence have cofibrant domains. The category $Top_*$ is not locally presentable, but if a combinatorial model category of spaces is used (e.g., $\Delta$-generated spaces), then the Joachim-Johnson model would satisfy the conditions of Theorem \ref{main}, by Lemma \ref{lemma:tractable-transfer}. A desirable localization is pointed out in the introduction to \cite{ostvaer}: namely, to study the stable $C^*$-homotopy category. Theorem \ref{main} provides a semi-model category whose homotopy category is the stable $C^*$-homotopy category built from the Joachim-Johnson model, analogously to \cite[Theorem 4.12]{ostvaer}. Analogously to \cite[Theorem 4.70]{ostvaer} (following \cite{hovey-spectra}), one can also build a monoidal semi-model structure for symmetric spectra in the Joachim-Johnson model, semi-model structures for modules and monoids \cite[Theorem 4.72]{ostvaer}, and other localizations such as the exact semi-model structure \cite[Theorem 3.32]{ostvaer}, the matrix invariant projective semi-model structure \cite[Theorem 3.54]{ostvaer}, or the homotopy invariant model structure \cite[Theorem 3.65]{ostvaer}. In addition to the study of Kasparov's $KK$-theory, these model structures have applications to the $E$-theory of Connes-Higson \cite[Remark 3.29]{ostvaer}. Analogously to \cite{chorny-white}, one can also build semi-model structures for homotopy functors between the model categories above \cite[Theorem 4.84]{ostvaer}, or the stable semi-model structure on functors \cite[Theorem 4.94]{ostvaer}. This is a first step towards applying Goodwillie calculus to $C^*$-algebras.
\end{example}

We conclude with an example about the connection between spaces and chain complexes.

\begin{example} \label{ex:richter}
In \cite{richter}, Richter and Shipley construct a chain of Quillen equivalences between commutative algebra spectra over $HR$ (where $R$ is a commutative ring), and $E_\infty$-monoids in unbounded chain complexes of $R$-modules. Doing so requires lifting the Dold-Kan equivalence to commutative monoids in symmetric sequences $C(sAb^\Sigma)\leftrightarrows C(ch^\Sigma)$, and then from symmetric sequences to symmetric spectra. As pointed out in \cite[Remark 6.4]{richter}, the positive model structure on $C(ch^\Sigma)$ is not left proper (but is tractable). However, there is a long history of lifting localizations from the level of chain complexes to the level of spectra \cite{bauer}, and so Theorem \ref{main} is an important first step to carry this program out for localizations of commutative $HR$-algebra spectra lifted from $C(ch^\Sigma)$. There are many interesting localization of chain complexes, catalogued in \cite{white-localization, white-yau-coloc}, that can be carried out for $C(ch^\Sigma)$ using Theorem \ref{main}.
\end{example}

\subsection{Enriched categories}

The setting of enriched categories provides several examples of desirable left Bousfield localizations in non-left proper settings, with potential applications to factorization algebras. We begin with an example about weakly enriched categories.

\begin{example} \label{ex:bacard}
In \cite{bacard-published}, Bacard introduced a model structure for the study of ``co-Segal categories,'' which are weakly enriched categories over a symmetric monoidal model category $\M$, satisfying a Segal-style weak equivalence rather than the usual composition law. In \cite{bacard}, Bacard sought to improve the theory of co-Segal categories to shift from non-unital weak $\M$-categories to unital precategories. To do so, Bacard needed what he called an {\em implicit Bousfield localization} (page 4 of \cite{bacard}), because his ``easy model structure'' on co-Segal precategories is not known to be left proper (but, it is tractable). As \cite{bacard} was never published, it is difficult to know if the implicit Bousfield localization worked. With Theorem \ref{main}, it is possible to achieve Bacard's goal, when $\M$ is a combinatorial model category with domains of the generating cofibrations cofibrant.

Bacard defines a co-Segal precategory with object set $X$ as a normal lax functor of 2-categories $\cat C: (\cat S_X)^{op} \to \M$, for a particular 2-category $\cat S_X$. The components are denoted $\cat C_{AB}: \cat S_X(A,B)^{op} \to \M$. Bacard is forced to discard his first notion of weak equivalence of co-Segal precategories, because it does not lead to a left proper model structure. This notion defines a weak equivalence to be a morphism $(\sigma, f): \cat C \to \cat D$ such that each natural transformation $\sigma_{AB}$ is a levelwise weak equivalence in $\cat M$. With Theorem \ref{main}, one could carry out the program of \cite{bacard} for either the ``easy model structure'' or for his notion of strict $\M$-categories, as long as $\M$ is a combinatorial model category with domains of the generating cofibrations cofibrant. Lastly, throughout \cite{bacard}, Bacard has statements that assume the existence of various left Bousfield localizations (e.g., 3.23, 3.44, 4.6, 4.14). Theorem \ref{main} can be used to verify that these localizations exist, and also to weaken the requirement that $\M$ be left proper.
\end{example}

We conclude with an application to the theory of dg-categories, i.e., categories enriched in chain complexes over a fixed commutative ring $k$.

\begin{example} \label{ex:dgcat}
The category $\dgcat(k)$ of small dg-categories, admits a tractable model structure \cite[Theorem 1.7]{tabuada} that is much used in the study of derived algebraic geometry \cite{toen}. This model category fails to be left proper in general \cite[Example 1.22]{tabuada}, unless strong conditions are placed on $k$. Nevertheless, To\"{e}n is able to construct localizations of its homotopy category \cite[Corollary 8.7]{toen} (building on earlier work of Drinfeld, Keller, and Lyubashenko).  Theorem \ref{main} allows us to lift To\"{e}n's localizations from the homotopy category level to the semi-model category level. Furthermore, Tabuada introduces a model structure on $\dgcat(k)$ with weak equivalences the pretriangulated equivalences \cite[Theorem 1.30]{tabuada} as a left Bousfield localization of the model structure above \cite[Proposition 1.33]{tabuada}, and then introduces the Morita model structure as a left Bousfield localization of the pretriangulated model structure \cite[Proposition 1.39]{tabuada}. This is done despite the fact that all of these model structures fail to be left proper. Using Theorem \ref{main}, we can actually achieve the Morita model structure as a left Bousfield localization (including the universal property), and \cite[Theorem 1.37]{tabuada} tells us the local semi-model structure is in fact a full model structure. Lastly, Theorem \ref{main} can be used to prove the existence of various localizations of the model structures above desired by Tabuada, e.g. Theorems 8.5, 8.17, and 8.25, Remark A.10, and Section 2.2.6 of \cite{tabuada}.
\end{example}

\subsection{Functor Calculus}

Another application of Theorem \ref{main} is to Goodwillie calculus for general model categories. The following example is motivated by Goodwillie's work studying categories of functors between categories of spaces and spectra. Our treatment follows Pereira \cite{pereira}, who seeks a version of Goodwillie calculus for functors between categories of algebras over operads. The main idea is to recast Goodwillie's $n$-excisive approximation as a left Bousfield localization, as was done previously by Biedermann, Chorny, and R\"{o}ndigs, but Pereira's setting is not left proper.

\begin{example} \label{ex:pereira}
Let $\cat O$ be a simplicial operad. Let $\cat C$ be a pointed simplicial model category such that the stable projective model structure on spectra $Sp(\algo(\cat C))$ exists. This occurs, for example, if $\cat C$ has domains of the generating cofibrations cofibrant. Even if $\cat C$ is simplicial sets or spectra, $\algo$ is almost never left proper, as Pereira shows (one case where is is left proper is if $\cat O$ is the Com operad \cite{white-commutative-monoids}).

Let $A$ and $B$ be ring spectra, $D$ a small subcategory of $A$-modules, and let $\M_1$ (resp. $\M_2$) be the category of simplicial functors $Fun(D,Mod_B)$ (resp. spectral functors). Instead of taking $D$ to be a small subcategory, one could alternatively study small functors $Fun^s(Mod_A,Mod_B)$, i.e. functors that are left Kan extensions of functors determined on a small subcategory, as is done in \cite{chorny-white}. Using Theorem \ref{main}, the projective model structures on $\M_i$ admit several left Bousfield localizations, where the new fibrant objects are homotopy functors, or linear functors, or $n$-excisive functors. For the latter, the localization $X\to L(X)$ is Goodwillie's $n$-excisive approximation.

Remark 4.11 in \cite{pereira} explains why having this localization on the model (or semi-model) level would be desirable. Pereira is able to prove an equivalence of homotopy categories, which can be lifted to an equivalence of semi-model categories using Theorem \ref{main}. An application is the characterization of the Goodwillie tower of the identity in $\algo$ as the homotopy completion tower associated to truncated operads $\cat O_{\leq n}$.

\end{example}

Other contexts where one might wish to extend the techniques of Goodwillie calculus (other than to $\algo$ as in the example above) include the setting of graph theory or the category of small categories. We discuss these settings now.

\begin{example} \label{ex:vicinsky}
In \cite{deb}, Vicinsky worked out the homotopy-theoretic foundations required to apply Goodwillie calculus to model categories of graphs and small categories. Traditionally, this requires a model structure $\M$ that is pointed, left proper, and simplicial \cite[Hypothesis 2.28]{deb}. Partially, this is required to build a stable model structure for spectra $Sp(\M)$. However, the model structure used by Vicinsky on pointed directed graphs (originally due to Bisson and Tsemo) is not left proper \cite[Proposition 5.8]{deb}. However, Theorem \ref{main} can be used to carry out the program laid out by Vicinsky: constructing a semi-model structure for spectra on graphs, then proving this category is homotopically trivial. Lastly, Theorem \ref{main} may be used to verify Vicinsky's Conjecture 6.10. Her model structures $Cat_n$ are transferred from $n$-truncated model structures on simplicial sets \cite[Proposition 6.8]{deb}, and hence are tractable by Lemma \ref{lemma:tractable-transfer}. Theorem \ref{main} provides a semi-model structure on spectra in $Cat_n$, and hence a framework to lift the Quillen equivalence $sSet_n \leftrightarrows Cat_n$ to categories of spectra, as required to prove \cite[Conjecture 6.10]{deb}.
\end{example}

 \subsection{Other possible directions}

We conclude with some suggestions about possible future applications  that don't fit into the categories above. 

In \cite{bergner}, Bergner defined the notion of a homotopy colimit of a diagram of model categories $\M_i$, as a quotient of a coproduct $\coprod M_i$. Dualizing her earlier work on homotopy limits, which she studied as a right Bousfield localization, the quotient required in a homotopy colimit can be studied as a left Bousfield localization. As Bergner points out, there are a number of technical difficulties, but Theorem \ref{main} can be used to circumvent the requirement that the coproduct (semi-)model structure be left proper, just as Barwick's right Bousfield localization \cite{barwickSemi} was used in the study of homotopy limits. The upside of this approach is that having a semi-model structure for the homotopy colimit gives more structure for computations than simply the relative category structure constructed in \cite{bergner}.

Another source of potential applications of Theorem \ref{main} would be to do left Bousfield localization after a right Bousfield localization, since right Bousfield localization often destroys left properness. This would occur, for example, if one wished to build spectra or do Goodwillie calculus for a model structure defined as a right Bousfield localization, such as model structures used for the study of slices in the slice spectral sequence, or $A$-cellular model structures in spaces, chain complexes, and categories \cite{white-yau-coloc}.

While it may seem that we have exhaustively cataloged all situations where one wishes to do left Bousfield localization without left properness, we are confident that there are in fact many more cases where Theorem \ref{main} will be useful. We also believe Theorem \ref{main-smith} will be useful in its own right.


\begin{thebibliography}{10}

\bibitem[AR94]{adamek} J. Adamek and J. Rosicky, \textit{Locally Presentable and Accessible Categories}, Cambridge University Press,  London Mathematical Society Lecture Note Series (189), 1994.

\bibitem[Bac13]{bacard} Hugo Bacard, Pursuing Lax Diagrams and Enrichment, preprint available as arXiv:1312.7833.

\bibitem[Bac14]{bacard-published} Hugo Bacard, 
 Toward weakly enriched categories: co-Segal categories,
\textit{Journal of Pure and Applied Algebra}, 218, pp.1130-1170, 2014.


\bibitem[BD95]{baez-dolan}
John~C. Baez and James Dolan.
\newblock Higher-dimensional algebra and topological quantum field theory.
\newblock {\em J. Math. Phys.}, 36(11):6073--6105, 1995.

\bibitem[Bar10]{barwickSemi}
Clark Barwick.
\newblock On left and right model categories and left and right {B}ousfield
  localizations.
\newblock {\em Homology, Homotopy Appl.}, 12(2):245--320, 2010.

\bibitem[BBPTY17a]{tillmann1} Maria Basterra, Irina Bobkova, Kate Ponto, Ulrike Tillmann, Sarah Yeakel \newblock Inverting operations in operads. \newblock \textit{Topology and its Applications}, 235 pp.130-145, 2017.

\bibitem[BBPTY17b]{tillmann2} Maria Basterra, Irina Bobkova, Kate Ponto, Ulrike Tillmann, Sarah Yeakel \newblock Infinite loop spaces from operads with homological stability. \newblock \textit{Advances in Mathematics}, 321 pp.391-430, 2017.

\bibitem[Bat10]{batanin-locally-constant} Michael Batanin. \newblock Locally constant $n$-operads as higher braided operads, \newblock {\em J. Noncommut. Geom.}, 4, pp.237-263, 2010.


\bibitem[Bat17]{batanin-baez-dolan-via-semi}
Michael Batanin.
\newblock An operadic proof of the {B}aez-{D}olan stabilisation hypothesis,
  {\em Proceedings of the AMS} 145, 2785-2798, 2017. 

\bibitem[BB17]{batanin-berger}
Michael Batanin and Clemens Berger.
\newblock Homotopy theory for algebras over polynomial monads, \textit{Theory and Application of Categories}, Vol. 32, No. 6, 148-253, 2017.


\bibitem[BD19]{florian} Michael Batanin and Florian De Leger, Polynomial monads and delooping of mapping spaces, {\em J. Noncommut. Geom.}, Volume 13, Issue 4, pp. 1521-1576, 2019.

\bibitem[BDW23a]{bdw1}
Michael Batanin, Florian De Leger, and David White.
\newblock Model structures on operads and algebras from a global perspective. Available as arXiv:2311.07320.

\bibitem[BDW23b]{bdw2}
Michael Batanin, Florian De Leger, and David White.
\newblock Quasi-tame substitudes and model structures on Grothendieck constructions. Available as arXiv:2311.07322.

\bibitem[BW15]{batanin-white-baez-dolan}
Michael Batanin and David White.
\newblock Baez-Dolan Stabilization via (Semi-)Model Categories of Operads, in ``Interactions between Representation Theory, Algebraic Topology, and Commutative Algebra,'' \textit{Research Perspectives CRM Barcelona}, Volume 5 (2015), pages 175-179, ed. Dolors Herbera, Wolfgang Pitsch, and Santiago Zarzuela, Birkh\"{a}user.



\bibitem[BW21]{batanin-white-eilenberg-moore}  M. Batanin and D. White, Left Bousfield localization and Eilenberg-Moore Categories, \textit{Homology, Homotopy and Applications}, vol. 23(2), pp.299-323, 2021. 


\bibitem[BW22]{Reedy-paper} Michael Batanin and David White. Homotopy theory of algebras of substitudes and their localisation. \textit{Transactions of the American Mathematical Society}, Volume 375, Number 5, May 2022, Pages 3569-3640.

\bibitem[Bau99]{bauer} Friedrich Bauer, The {B}oardman category of spectra, chain complexes and \newblock 
(co-)localizations, \textit{Homology Homotopy Appl}. (1), 95-116, 1999.

\bibitem[Bea19]{beardsley} Jonathan Beardsley, A Theorem on Multiplicative Cell Attachments with an Application to Ravenel's X(n) Spectra, \textit{Journal of Homotopy and Related Structures} 14-3 (2019), 611-624.

\bibitem[Bek00]{beke} Tibor Beke, Sheafifiable homotopy model categories, \textit{Mathematical Proceedings of the Cambridge Philosophical Society}, Volume 129, Issue 3, 2000, pp. 447-475.

\bibitem[BM07]{bm07} Clemens Berger and Ieke Moerdijk, Resolution of coloured operads and rectification of homotopy algebras, Categories in algebra, geometry and mathematical physics, \textit{Contemporary Mathematics} 431 (American Mathematical Society, Providence, RI, 2007) 31-58.

\bibitem[Ber14]{bergner} Julie Bergner, Homotopy colimits of model categories, An alpine expedition through algebraic topology, \textit{Contemp. Math.}, 617 (2014), 31-37.

\bibitem[BH22]{bourke-henry} John Bourke and Simon Henry, Algebraically cofibrant and fibrant objects revisited, \textit{Homology, Homotopy and Applications}, volume 24, issue 1 (2022), pp.271-298.


\bibitem[Bou75]{bousfield-localization-spaces-wrt-homology}
A.~K. Bousfield.
\newblock The localization of spaces with respect to homology.
\newblock {\em Topology}, 14:133--150, 1975.

\bibitem[Bou79]{bous79}
A.~K. Bousfield.
\newblock The localization of spectra with respect to homology.
\newblock {\em Topology}, 18(4):257--281, 1979.

\bibitem[BCL18]{lazarev} Christopher Braun, Joseph Chuang, Andrey Lazarev, Derived localisation of algebras and modules, \textit{Advances in Mathematics}, volume 328, issue 68, 555-622, 2018.

\bibitem[CR14]{ching-riehl} Michael Ching and Emily Riehl, Coalgebraic models for combinatorial model categories, \textit{Homology, Homotopy and Applications}, volume 16, issue 2 (2014), pp.171-184.

\bibitem[Cis06]{Cis06} Denis-Charles Cisinski, Les pr\'{e}faisceaux comme mod\'{e}les des types d'homotopie (French, with English and French summaries), \textit{Ast\'{e}risque} 308, xxiv+390, 2006.

\bibitem[Cis09]{cisinski-locally-constant} Denis-Charles Cisinski. Locally constant functors. \textit{Math. Proc. Cambridge Philos. Soc.} \textbf{147}, 593-614, 2009.

\bibitem[CW18]{chorny-white} Boris Chorny and David White. A variant of a Dwyer-Kan theorem for model categories, available as arXiv:1805.05378, to appear in \textit{Algebraic \& Geometric Topology}.

\bibitem[Dug01]{dugger} D. Dugger, Combinatorial Model Categories Have Presentations, 
\textit{Advances in Mathematics}, 
Volume 164, Issue 1, Pages 177-201, 2001.

\bibitem[Dwy06]{dwyerLoc} William G. Dwyer, Noncommutative localization in homotopy theory, in \textit{Noncommutative Localization in Algebra and Topology}, Cambridge University Press, 2006, pp 24-39.

\bibitem[EKMM97]{EKMM}
A.~D. Elmendorf, I.~Kriz, M.~A. Mandell, and J.~P. May.
\newblock {\em Rings, modules, and algebras in stable homotopy theory},
  volume~47 of  Mathematical Surveys and Monographs.
\newblock American Mathematical Society, Providence, RI, 1997.
\newblock With an appendix by M. Cole.

\bibitem[Fre09]{fresse-book}
B. Fresse.
\newblock {\em Modules over operads and functors}, volume 1967 of {\em Lecture Notes in Mathematics}.
\newblock Springer-Verlag, Berlin, 2009.

\bibitem[Gil16]{gillespie-survey} 
J. Gillespie. Hereditary abelian model categories, \textit{Bulletin of the London Mathematical Society} 48, no. 6, 895-922, 2016.

\bibitem[GH04]{goerss-hopkins} P. G. Goerss and M. J. Hopkins, Moduli problems for structured ring spectra, available at 
\newblock https://sites.math.northwestern.edu/$\sim$pgoerss/spectra/obstruct.pdf, 2004.

\bibitem[GRS{\O}12]{grso}
Javier~J. Guti{\'e}rrez, Oliver R{\"o}ndigs, Markus Spitzweck, and Paul~Arne
  {\O}stv{\ae}r.
\newblock Motivic slices and colored operads.
\newblock {\em Journal of Topology}, 5:727--755, 2012.

\bibitem[GW18]{gutierrez-white-equivariant}
Javier~J. Guti{\'e}rrez and David White.
\newblock Encoding equivariant commutativity via operads, \textit{Algebraic \& Geometric Topology}, Volume 18, Number 5, pp.2919-2962, 2018.


\bibitem[HRY17]{hackney} Phil Hackney, Marcy Robertson, and Donald Yau. 
Relative left properness of colored operads,  \textit{Algebr. Geom. Topol.}, Volume 16, Number 5 (2016), 2691-2714.

\bibitem[HZ19]{harper-zhang} John Harper and Yu Zhang, 
Topological Quillen localization of structured ring spectra, \textit{Tbilisi Math. J.}, Volume 12, Issue 3 (2019), 69-91.

\bibitem[Hen23]{henry} Simon Henry, Combinatorial and accessible weak model categories, \textit{Journal of Pure and Applied Algebra}, Volume 227, Issue 2 (2023).


\bibitem[Hir03]{hirschhorn}
Philip~S. Hirschhorn.
\newblock {\em Model categories and their localizations}.
\newblock Mathematical Surveys and Monographs. American Mathematical Society,
  Providence, RI, 2003.

\bibitem[Hov98]{hovey-monoidal}
Mark Hovey.
\newblock Monoidal model categories, preprint available electronically from
  http://arxiv.org/abs/math/9803002.
\newblock 1998.

\bibitem[Hov99]{hovey-book}
Mark Hovey.
\newblock {\em Model categories}, volume~63 of {\em Mathematical Surveys and
  Monographs}.
\newblock American Mathematical Society, Providence, RI, 1999.

\bibitem[Hov01]{hovey-spectra}
Mark Hovey.
\newblock Spectra and symmetric spectra in general model categories.
\newblock {\em J. Pure Appl. Algebra}, 165(1):63--127, 2001.

\bibitem[Hov02]{hovey-cotorsion} M. Hovey, Cotorsion pairs, model category structures, and representation theory, \textit{Math Z.} 241 (3), 553--592, 2002.

\bibitem[IJ02]{johnson} Michele Intermont, Mark Johnson, Model structures on the category of ex-spaces, \textit{Topology and its Applications}, Volume 119, Issue 3, 30 April 2002, Pages 325-353.

\bibitem[JJ07]{kk} Michael Joachim, Mark Johnson, Realizing Kasparov's $KK$-theory groups as the homotopy classes of maps of a Quillen model category, \textit{Contemporary Math} \#399: An Alpine Anthology of Homotopy Theory, Dominique Arlettaz and Katherine Hess editors.

\bibitem[Lur09]{lurie-htt} Jacob Lurie, \textit{Higher Topos Theory}, Annals of Mathematics Studies, book 170, Princeton University Press, 2009.

\bibitem[Man01]{mandell} Mike Mandell, $E_\infty$ algebras and $p$-adic homotopy theory, \textit{Topology}, Volume 40, Issue 1, January 2001, Pages 43-94.

\bibitem[Nui19]{nuiten} Joost Nuiten, Homotopical Algebra for Lie Algebroids, \textit{Applied Categorical Structures}, Volume 27, Issue 5 (2019), pp 493-534.

\bibitem[{\O}st10]{ostvaer} Paul Arne {\O}stv{\ae}r, \textit{Homotopy Theory of $C^*$-Algebras}, Birkhauser Basel, 2010.


\bibitem[Per17]{pereira} Lu\'{i}s Alexandre Pereira, Goodwillie calculus in the category of algebras over a spectral operads, preprint available from https://math.mit.edu/$\sim$luisalex/GoodwillieCalculus$\%$20inAlgO.pdf, accessed June 1, 2019.

\bibitem[Rav84]{rav84} Doug Ravenel, Localization with Respect to Certain Periodic Homology Theories, \textit{American Journal of Mathematics}, Vol. 106, No. 2, (Apr., 1984), pp. 351-414.

\bibitem[Rez10]{Rezk} C. Rezk, A Cartesian presentation of weak $n$-categories, \textit{Geom. Topol.} 14 (2010) 521-571.

\bibitem[RS17]{richter} Birgit Richter and Brooke Shipley, An algebraic model for commutative $H\mathbb{Z}$-algebras, \textit{Algebr. Geom. Topol.}, Volume 17, Number 4 (2017), 2013-2038.

\bibitem[Spi01]{spitzweck-thesis} Markus Spitzweck. \newblock Operads, algebras and modules in general model categories, preprint  available electronically from http://arxiv.org/abs/math/0101102.
\newblock 2001.

\bibitem[Tab15]{tabuada} Goncalo Tabuada, \textit{Noncommutative Motives}, AMS University lecture series 63, 2015.

\bibitem[To\"{e}10]{toen} Bertrand To\"{e}n, Homotopy theory of dg-categories and derived Morita theory, \textit{Inventiones mathematicae} volume 167, pages 615-667, 2007.

\bibitem[Vic15]{deb} Deborah Vicinsky, The homotopy calculus of categories and graphs, Ph.D. thesis, available electronically from \\
https://scholarsbank.uoregon.edu/xmlui/bitstream/handle/1794/19283/\\
Vicinsky$\_$oregon$\_$0171A$\_$11298.pdf?sequence=1, 2015.

\bibitem[Voe10]{voevodsky} Vladimir Voevodsky, Simplicial radditive functors, \textit{Journal of K-Theory}, 5 (2010), 201-244.


\bibitem[Whi17]{white-commutative-monoids}
David White.
\newblock Model structures on commutative monoids in general model categories. \newblock {\em Journal of Pure and Applied Algebra}, Volume 221, Issue 12, 2017, Pages 3124-3168. Available as arXiv:1403.6759.

\bibitem[Whi14]{white-thesis}
David White.
\newblock Monoidal {B}ousfield localizations and algebras over operads.
\newblock 2014.
\newblock Thesis (Ph.D.)--Wesleyan University.

\bibitem[Whi21a]{white-localization}
David White.
\newblock Monoidal {B}ousfield localizations and algebras over operads, Equivariant Topology and Derived Algebra, Cambridge University Press (2021), 179-239. 

\bibitem[Whi21b]{white-oberwolfach} David White. Substitudes, Bousfield localization, higher braided operads, and Baez-Dolan stabilization, \textit{Mathematisches Forschungsinstitut Oberwolfach}, Number 46, 2021: Homotopical Algebra and Higher Structures.





\bibitem[WY23]{white-yau4}
David White and Donald Yau. Right Bousfield Localization and Eilenberg-Moore Categories, \textit{Higher Structures} 7(1):22–39, 2023.


\bibitem[WY17]{white-yau6} David White and Donald Yau. Smith Ideals of Operadic Algebras in Monoidal Model Categories, available as arXiv:1703.05377, to appear in \textit{Algebraic \& Geometric Topology}.

\bibitem[WY18]{white-yau}
David White and Donald Yau.
\newblock {B}ousfield localizations and algebras over colored operads, 
\newblock {\em Applied Categorical Structures}, 26:153--203, 2018.
Available as arxiv:1503.06720.


\bibitem[WY19]{white-yau3} David White and Donald Yau, Homotopical adjoint lifting theorem, \textit{Applied Categorical Structures}, 27:385-426, 2019.

\bibitem[WY20]{white-yau-coloc}
David White and Donald Yau.
\newblock Right {B}ousfield localization and operadic algebras, \textit{Tbilisi Math. Journal}, Special issue (HomotopyTheorySpectra - 2020), 71-118, 2020. 


\bibitem[Yau19]{yau-book} Donald Yau, \textit{Homotopical Quantum Field Theory}, World Scientific, 2019.

\end{thebibliography}
\end{document}